\documentclass[reqno]{amsart}
\usepackage{amsmath, amsthm, amssymb, amstext}

\usepackage[left=3.5cm,right=3.5cm,top=1.5cm,bottom=1.5cm,includeheadfoot]{geometry}
\usepackage{hyperref,xcolor}
\hypersetup{
 pdfborder={0 0 0},
 colorlinks,
}
\usepackage{enumitem}
\allowdisplaybreaks

\newtheorem{theorem}{Theorem}
\newtheorem{remark}[theorem]{Remark}
\newtheorem{lemma}[theorem]{Lemma}
\newtheorem{proposition}[theorem]{Proposition}

\newtheorem{definition}[theorem]{Definition}
\newtheorem{example}[theorem]{Example}

\DeclareMathOperator*{\divergenz}{div}              %
\DeclareMathOperator*{\supp}{supp}         %
\DeclareMathOperator*{\essinf}{ess ~inf}         %
\DeclareMathOperator*{\esssup}{ess ~sup}         %
\newcommand{\Ss}{\textnormal{(S}_+\textnormal{)}}
\DeclareMathOperator*{\diam}{diam}

\newcommand{\N}{\mathbb{N}}

\newcommand{\R}{\mathbb{R}}

\newcommand{\RN}{\mathbb{R}^N}

\newcommand{\Lp}[1]{L^{#1}(\Omega)}

\newcommand{\Lprand}[1]{L^{#1}(\partial\Omega)}
\newcommand{\Wp}[1]{W^{1,#1}(\Omega)}

\newcommand{\Wpzero}[1]{W^{1,#1}_0(\Omega)}

\newcommand{\lan}{\langle}
\newcommand{\ran}{\rangle}
\newcommand{\eps}{\varepsilon}
\newcommand{\ph}{\varphi}

\newcommand{\into}{\int_{\Omega}}

\newcommand{\weak}{\rightharpoonup}

\newcommand{\Linf}{L^{\infty}(\Omega)}
\newcommand{\close}{\overline{\Omega}}

\newcommand{\cprime}{$'$}

\renewcommand{\l}{\left}
\renewcommand{\r}{\right}
\numberwithin{theorem}{section}
\numberwithin{equation}{section}

\newcommand*\diff{\mathop{}\!\mathrm{d}}

\newcommand{\abs}[1]{\left\lvert #1 \right\rvert}
\newcommand{\norm}[1]{\left\lVert #1 \right\rVert}

\newcommand{\W}{W^{1, \mathcal{H}}_0(\Omega)}

\title[A new class of double phase variable exponent problems]{A new class of double phase variable exponent problems: Existence and uniqueness}

\author[\'{A}.\,Crespo-Blanco]{\'{A}ngel Crespo-Blanco}
\address[\'{A}.\,Crespo-Blanco]{Technische Universit\"{a}t Berlin, Institut f\"{u}r Mathematik, Stra\ss e des 17.\,Juni 136, 10623 Berlin, Germany}
\email{crespo@math.tu-berlin.de}

\author[L.\,Gasi\'nski]{Leszek Gasi\'nski}
\address[L.\,Gasi\'nski]{Pedagogical University of Cracow, Department of Mathematics, Podchorazych 2, 30-084 Cracow, Poland}
\email{leszek.gasinski@up.krakow.pl}

\author[P.\,Harjulehto]{Petteri Harjulehto}
\address[P.\,Harjulehto]{Department of Mathematics and Statistics, FI-20014 University of Turku, Finland}
\email{petteri.harjulehto@utu.fi}

\author[P.\,Winkert]{Patrick Winkert}
\address[P.\,Winkert]{Technische Universit\"{a}t Berlin, Institut f\"{u}r 
Mathematik, Stra\ss e des 17.\,Juni 136, 10623 Berlin, Germany}
\email{winkert@math.tu-berlin.de}

\subjclass{35J15, 35J62, 35P30, 47B92, 47H05}
\keywords{Density of smooth functions, double phase operator with variable exponent, convection term, Musielak-Orlicz Sobolev space, existence results, uniqueness}

\begin{document}

\begin{abstract}
	In this paper we introduce a new class of quasilinear elliptic equations driven by the so-called double phase operator with variable exponents. We prove certain properties of the corresponding Musielak-Orlicz Sobolev spaces (an equivalent norm, uniform convexity, Radon-Riesz property with respect to the modular) and the properties of the new double phase operator (continuity, strict monotonicity, $\Ss$-property). In contrast to the known constant exponent case we are able to weaken the assumptions on the data. Finally we show the existence and uniqueness of corresponding elliptic equations with right-hand sides that have gradient dependence (so-called convection terms) under very general assumptions on the data. As a result of independent interest, we also show the density of smooth functions in the new Musielak-Orlicz Sobolev space even when the domain is unbounded.
\end{abstract}
	
\maketitle
	
\section{Introduction}

Given a bounded domain $\Omega \subseteq \R^N$, $N \geq 2$, with Lipschitz boundary $\partial \Omega$, this paper is concerned with a new double phase operator with variable exponents given by
\begin{align}\label{double-phase-variable}
	\divergenz\left(|\nabla u|^{p(x)-2}\nabla u+\mu(x) |\nabla u|^{q(x)-2}\nabla u\right)
\end{align}
with $p,q \in C(\close)$ such that $1<p(x)<N$, $p(x)<q(x)$ for all $x \in 
\close$ and $0\leq \mu(\cdot) \in \Lp{1}$. This operator is the natural extension of the classical double phase operator when $p$ and $q$ are constants, namely 
\begin{align}\label{double-phase}
	\divergenz\Big(|\nabla u|^{p-2}\nabla u+\mu(x) |\nabla u|^{q-2}\nabla u\Big).
\end{align}
It is clear that when  $\inf_{\close} \mu>0$ or $\mu\equiv 0$, then the operator in \eqref{double-phase-variable} becomes the weighted $(q(x),p(x))$-Laplacian or the $p(x)$-Laplacian, respectively. The energy functional $I\colon\W\to \R$ related to the double phase 
operator \eqref{double-phase-variable} is given by
\begin{align}\label{integral_minimizer}
	I(u)=\into \l( \frac{|\nabla u|^{p(x)}}{p(x)} 
	+ \mu(x) \frac{|\nabla u|^{q(x)}}{q(x)}\r)\diff x,
\end{align}
where the integrand $H(x,\xi)=\frac{1}{p(x)}|\xi|^{p(x)}+\frac{\mu(x)}{q(x)}|\xi|^{q(x)}$ for all $(x,\xi) \in \Omega\times \R^N$ of $I$ has unbalanced growth if $0\leq \mu(\cdot) \in \Lp{\infty}$, that is,
\begin{align*}
	b_1|\xi|^{p(x)} \leq H(x,\xi) \leq b_2 \l(1+|\xi|^{q(x)}\r)\quad \text{for a.\,a.\,}x\in\Omega \text{ and for all }\xi\in\R^N \text{ with }b_1,b_2>0.
\end{align*}

The main characteristic of the functional $I$ is the change of ellipticity on the set where the weight function is zero, that is, on the set $\{x\in \Omega: \mu(x)=0\}$. Indeed, the energy density of $I$ exhibits ellipticity in the gradient of order $q(x)$ in the set $\{x\in\Omega\,:\,\mu(x)>\eps\}$ for any fixed $\eps>0$ and of order $p(x)$ on the points $x$ where $\mu(x)$ vanishes. So the integrand $H$ switches between two different phases of elliptic behaviours. This is the reason why it is called double phase. 

Zhikov \cite{Zhikov-1986} was the first who studied functionals whose integrands change their ellipticity according to a point in order to provide 
models for strongly anisotropic materials. Functionals of the form \eqref{integral_minimizer} have been studied by several authors with respect to 
regularity of local minimizers (also for nonstandard growth). We refer to 
the works of Baroni-Colombo-Mingione \cite{Baroni-Colombo-Mingione-2015,Baroni-Colombo-Mingione-2016,Baroni-Colombo-Mingione-2018}, Baroni-Kuusi-Mingione \cite{Baroni-Kuusi-Mingione-2015}, Byun-Oh \cite{Byun-Oh-2020}, Colombo-Mingione \cite{Colombo-Mingione-2015a,Colombo-Mingione-2015b}, De Filippis \cite{De-Filippis-2018}, De Filippis-Palatucci \cite{De-Filippis-Palatucci-2019}, Harjulehto-H\"{a}st\"{o}-Toivanen \cite{Harjulehto-Hasto-Toivanen-2017}, 
Marcellini \cite{Marcellini-1991,Marcellini-1989b}, Ok \cite{Ok-2018,Ok-2020}, Ragusa-Tachikawa \cite{Ragusa-Tachikawa-2016,Ragusa-Tachikawa-2020} 
and the references therein. Moreover, recent results for nonuniformly elliptic variational problems and nonautonomous functionals can be found in the papers of Beck-Mingione \cite{Beck-Mingione-2020,Beck-Mingione-2019}, 
De Filippis-Mingione \cite{De-Filippis-Mingione-2020} and H\"{a}st\"{o}-Ok \cite{Hasto-Ok-2019}.

In general, double phase differential operators and corresponding energy functionals given in \eqref{double-phase}, \eqref{double-phase-variable} and \eqref{integral_minimizer}, respectively, appear in several physical applications. For example, in the elasticity theory, the modulating coefficient $\mu(\cdot)$ dictates the geometry of composites made of two different materials with distinct power hardening exponents $q(x)$ and $p(x)$, 
see Zhikov \cite{Zhikov-2011}. We also refer to other applications which can be found in the works of
Bahrouni-R\u{a}dulescu-Repov\v{s} \cite{Bahrouni-Radulescu-Repovs-2019} on transonic flows, Benci-D'Avenia-Fortunato-Pisani \cite{Benci-DAvenia-Fortunato-Pisani-2000} on quantum physics and  Cherfils-Il\cprime yasov \cite{Cherfils-Ilyasov-2005} on reaction diffusion systems.

In this paper we study first the corresponding function space related to the given double phase operator with variable exponents given in \eqref{double-phase-variable}. This leads to Musielak-Orlicz Sobolev spaces which 
turn out to be reflexive Banach spaces. Under the condition that the weight function $\mu(\cdot)$ is bounded we also show that
\begin{align*}
	\inf \l \{\lambda>0 \, : \, \into \l[\l(\frac{|\nabla u|}{\lambda}\r)^{p(x)}+\mu(x)\l(\frac{|\nabla u|}{\lambda}\r)^{q(x)}\r] \diff x \leq 1 \r\}
\end{align*} 
is an equivalent norm in $\W$ under the additional assumption that
\begin{equation*}
	q(x) < p^*(x) \quad \text{for all } x \in \close.
\end{equation*}
This condition (taking constant exponents) is weaker than the usual one for the constant exponent double phase setting, namely $\mu(\cdot)$ is Lipschitz continuous and
\begin{align}\label{condition_equivalent_norm2}
	\frac{q}{p}<1+\frac{1}{N},
\end{align}
see Colasuonno-Squassina \cite[Proposition 2.18(iv)]{Colasuonno-Squassina-2016}. In this direction we also make use of its natural extension
\begin{align}\label{condition_equivalent_norm}
	\frac{q_+}{p_-}<1+\frac{1}{N}
\end{align}
with $q_+$ being the maximum of $q$ and $p_-$ being the minimum of $p$ on 
$\close$, in order to prove another compact embedding result and the density of smooth functions. Condition \eqref{condition_equivalent_norm2} was 
used for the first time by Baroni-Colombo-Mingione \cite[see (1.8)]{Baroni-Colombo-Mingione-2015} in order to obtain regularity results of local minimizers for double phase integrals, see also the related works \cite{Baroni-Colombo-Mingione-2016} and \cite{Baroni-Colombo-Mingione-2018} of the same authors and Colombo-Mingione \cite{Colombo-Mingione-2015a}, \cite{Colombo-Mingione-2015b}. The condition is needed for the density of smooth functions. We are able to prove the same result under the condition \eqref{condition_equivalent_norm} and Lipschitz continuity on $p,q$ and $\mu$ in this variable exponent setting, see Theorem \ref{theorem_density_1}. 
Since the proof of Theorem \ref{theorem_density_1} does not need the boundedness of $\Omega$, the results holds true for unbounded domains, see Theorem \ref{theorem_density_1_unbounded}. In addition, we give a different 
proof for the density for unbounded domains under weaker conditions, namely, the exponents $p,q$ are bounded, log-H\"older continuous satisfying the log-H\"older decay condition and $q$ is $\frac{\alpha}{q_-}$-H\"older continuous while $\mu$ is $\alpha$-H\"older continuous such that
\begin{align*}
	\frac{q(x)}{p(x)} \le 1 + \frac{\alpha}{N}.
\end{align*}
In this case we do not need to suppose Lipschitz continuity on $p,q$ and $\mu$,  see Theorem \ref{theorem_density_2}.

After having the functional setting, we prove the properties of the new variable exponent double phase operator. It turns out that the operator is 
continuous, bounded, strictly monotone and satisfies the $\Ss$-property which is an important property when dealing with existence results of corresponding equations. 

In particular, we extend the results of Colasuonno-Squassina \cite{Colasuonno-Squassina-2016} concerning the properties of the function space as well as the related embeddings and of Liu-Dai \cite{Liu-Dai-2018} with respect to the properties of the operator to the variable exponent case and we are able to weaken the conditions on the data. So the results in \cite{Colasuonno-Squassina-2016} and \cite{Liu-Dai-2018} hold now under weaker assumptions.

Finally, we consider the existence and uniqueness of the following quasilinear elliptic equations 
\begin{equation}\label{problem8}
    \begin{aligned}
	-\divergenz\left(|\nabla u|^{p(x)-2}\nabla u+\mu(x) |\nabla u|^{q(x)-2}\nabla u\right) & =f(x,u,\nabla u)\quad && \text{in } \Omega,\\
	u & = 0 &&\text{on } \partial \Omega,
    \end{aligned}
\end{equation}
where $f\colon\Omega\times\R\times\R^N\to \R$ is a Carath\'{e}odory function, that is, $x\mapsto f(x,s,\xi)$ is measurable for all $(s,\xi)\in\R\times\R^N$ and $(s,\xi)\mapsto f(x,s,\xi)$ is continuous for a.a.\,$x\in\Omega$. Due to the gradient dependence of $f$ (often called convection term), problem \eqref{problem8}  does not have variational structure, so variational methods cannot be applied. Under a typical growth rate and a minor coercivity condition of $f$ we show the existence of at least one nontrivial weak solution to problem \eqref{problem8} which depends on the first eigenvalue of the $p_-$-Laplacian. Under an additional hypothesis we are also in the position to show uniqueness of the solution in case $ 2 \equiv p(x)<q(x) $ for all $x \in \close$.

To the best of our knowledge, this is the first work dealing with the variable exponent double phase operator given in the general form \eqref{double-phase-variable}. Let us mention some relevant papers in this direction. In 2018, Zhang-R\u{a}dulescu \cite{Zhang-Radulescu-2018} studied the following variable exponent elliptic equation
\begin{align}\label{intro-1}
	-\divergenz \mathrm{A}(x,\nabla u)+V(x)|u|^{\alpha(x)-2}u=f(x,u),
\end{align}
where $\mathrm{A}$ satisfies $p(x)$-structure conditions different from the double phase operator. Under appropriate conditions it is shown that problem \eqref{intro-1} has a pair of nontrivial constant sign solutions and infinitely many solutions, respectively. A similar setting can be found in the paper of Shi-R\u{a}dulescu-Repov\v{s}-Zhang \cite{Shi-Radulescu-Repovs-Zhang-2020}. Existence of a solution for the Baouendi–Grushin operator with convection term has been recently proved by Bahrouni-R\u{a}dulescu-Winkert \cite{Bahrouni-Radulescu-Winkert-2020} who studied the problem
\begin{equation*}
	\begin{aligned}
		-\Delta_{G(x,y)}u+A(x,y)(|u|^{G(x,y)-1}+|u|^{G(x,y)-3})u& =f\left((x,y),u,\nabla u\right) && \text{in } \Omega,\\
		u&= 0  &&\text{on } \partial \Omega,
	\end{aligned}
\end{equation*}
with
\begin{align*}
	A(x,y)=|\nabla_{x} G(x,y)|+|x|^{\gamma}|\nabla_{y} G(x,y)|\quad \text{for all $(x,y)\in \Omega$}.
\end{align*}
Here, $G\colon\close \to (1,\infty) $ is supposed to be a continuous function and $\Delta_{G(x,y)}$ stands for the Baouendi-Grushin operator with variable coefficient, see also the work of Bahrouni-R\u{a}dulescu-Repov\v{s} \cite{Bahrouni-Radulescu-Repovs-2019}.

Very recently, Arora-Shmarev \cite{Arora-Shmarev-2020} (see also Arora \cite{Arora-2021}) treated a parabolic problem of double phase type with variable growth of the form
\begin{align*}
	u_t-\divergenz\left(|\nabla u|^{p(x)-2}\nabla u+a(x) |\nabla u|^{q(x)-2}\nabla u\right)=F(x,u) \quad\text{in }Q_T=\Omega\times (0,T)
\end{align*}
with
\begin{align*}
	\frac{2N}{N+2}<p^-\leq p(x)\leq q(x)<p(x)+\frac{r}{2}
\end{align*}
and 
\begin{align*}
	0<r<r^*=\frac{4p^-}{2N+p^-(N+2)}, \quad p^-=\min_{\overline{Q}_T}p(x).
\end{align*}
Under certain conditions on the right-hand side the existence of a unique 
strong solution with a certain kind of regularity is shown. Finally, we refer to some works dealing with existence results for variable exponent problems defined in usual Sobolev spaces with variable exponents, see, for 
example, Cencelj-R\u{a}dulescu-Repov\v{s} \cite{Cencelj-Radulescu-Repovs-2018}, Gasi\'nski-Papageorgiou \cite{Gasinski-Papageorgiou-2011} and the references therein.

Existence results for double phase problems with constant exponents have been shown by several authors within the last five years. The corresponding eigenvalue problem of the double phase operator with Dirichlet boundary condition has been studied by Colasuonno-Squassina \cite{Colasuonno-Squassina-2016} who proved the existence and properties of related variational eigenvalues. Perera-Squassina \cite{Perera-Squassina-2018} showed the existence of a solution by applying Morse theory where they used a cohomological local splitting to get an estimate of the critical groups at zero. Multiplicity results including sign-changing solutions have been obtained by Gasi\'nski-Papa\-georgiou \cite{Gasinski-Papageorgiou-2019},  Liu-Dai \cite{Liu-Dai-2018} and Gasi\'nski-Winkert \cite{Gasinski-Winkert-2021} via the Nehari manifold treatment due to the lack of regularity results for such problems.

Other existence results for double phase problems based on truncation and 
comparison techniques can be found in the papers of Fiscella \cite{Fiscella-2020} (Hardy potentials), Fiscella-Pinamonti \cite{Fiscella-Pinamonti-2020} (Kirchhoff type problem), Gasi\'nski-Winkert \cite{Gasinski-Winkert-2020a,Gasinski-Winkert-2020b} (parametric and convection problems), Papageorgiou-R\u{a}dulescu-Repov\v{s} \cite{Papageorgiou-Radulescu-Repovs-2020a} (ground state solutions), Zeng-Bai-Gasi\'nski-Winkert \cite{Zeng-Bai-Gasinski-Winkert-2020, Zeng-Gasinski-Winkert-Bai-2020} (multivalued obstacle problems) and the references therein. For related works dealing with certain types of double phase problems we refer to the works of Barletta-Tornatore \cite{Barletta-Tornatore-2021}, Biagi-Esposito-Vecchi \cite{Biagi-Esposito-Vecchi-2021}, Farkas-Winkert \cite{Farkas-Winkert-2021}, Liu-Winkert \cite{Liu-Winkert-2022}, Papageorgiou-R\u{a}dulescu-Repov\v{s} \cite{Papageorgiou-Radulescu-Repovs-2019b} and R\u{a}dulescu \cite{Radulescu-2019}.

The paper is organized as follows. In Section \ref{section_2} we introduce the new Musielak-Orlicz Sobolev space, prove its properties already mentioned above and we will recall some basic facts about the spectrum of the $r$-Laplacian ($r\in(1,\infty)$) as well as definitions from the theory 
of monotone operators. Section \ref{section_3} is devoted to the properties of the new double phase operator and finally, in Section \ref{section_4}, we present our existence and uniqueness result for problem \eqref{problem8}.

\section{A new Musielak-Orlicz Sobolev space and some preliminaries}\label{section_2}

In this section we recall some known results and introduce a new function 
space needed in our approach providing some of its properties.

In the study of equations with variable exponent double phase phenomena we need 
to recall the definition of Lebesgue and Sobolev spaces with variable exponents. Most of the results can be found in the book of Diening-Harjulehto-H\"{a}st\"{o}-R$\mathring{\text{u}}$\v{z}i\v{c}ka \cite{Diening-Harjulehto-Hasto-Ruzicka-2011}, see also Fan-Zhao \cite{Fan-Zhao-2001}, Kov{\'a}{\v{c}}ik-R{\'a}kosn{\'{\i}}k \cite{Kovacik-Rakosnik-1991} and R\u{a}dulescu-Repov\v{s} \cite{Radulescu-Repovs-2015}. We will present them in a less general setting that matches our purpose.

Suppose that $\Omega$ is a bounded domain in $\mathbb{R}^N$ with Lipschitz boundary $\partial\Omega$ and let $r\in C_+(\close)$, where
\begin{align*}
	C_+(\close)=\big\{h \in C(\close) \, : \, 1<h(x) \text{ for all }x\in \close\big\}.
\end{align*}
For any $r\in C_+(\close)$ we define
\begin{align*}
	r_-=\min_{x\in \close}r(x) \quad\text{and}\quad r_+=\max_{x\in\close} r(x).
\end{align*}
Let $M(\Omega)$ be the space of all measurable functions $u\colon \Omega\to\R$. We identify two such functions when they differ only on a Lebesgue-null set. Then, for a given $r \in C_+(\close)$, the variable exponent Lebesgue space $\Lp{r(\cdot)}$ is defined as
\begin{align*}
	\Lp{r(\cdot)}=\l\{u \in M(\Omega)\,:\, \into |u|^{r(x)}\diff x<\infty \r\}
\end{align*}
equipped with the Luxemburg norm given by
\begin{align*}
	\|u\|_{r(\cdot)} =\inf \l \{\lambda>0 \, : \, \into \l(\frac{|u|}{\lambda}\r)^{r(x)}\diff x \leq 1 \r\}.
\end{align*}
It is clear that $(\Lp{r(\cdot)},\|\cdot\|_{r(\cdot)})$ is a separable and reflexive Banach space.  Let $r' \in C_+(\close)$ be the conjugate variable exponent to $r$, that is,
\begin{align*}
	\frac{1}{r(x)}+\frac{1}{r'(x)}=1 \quad\text{for all }x\in\close.
\end{align*}
We know that $\Lp{r(\cdot)}^*=\Lp{r'(\cdot)}$ and H\"older's inequality 
holds, that is,
\begin{align*}
	\into |uv| \diff x \leq \l[\frac{1}{r_-}+\frac{1}{r'_-}\r] \|u\|_{r(\cdot)}\|v\|_{r'(\cdot)} \leq 2 \|u\|_{r(\cdot)}\|v\|_{r'(\cdot)}
\end{align*}
for all $u\in \Lp{r(\cdot)}$ and for all $v \in \Lp{r'(\cdot)}$.

If $r_1, r_2\in C_+(\close)$ and $r_1(x) \leq r_2(x)$ for all $x\in \close$, then we have the continuous embedding
\begin{align*}
	\Lp{r_2(\cdot)} \hookrightarrow \Lp{r_1(\cdot)}.
\end{align*}

The corresponding variable exponent Sobolev spaces can be defined in the same way using the variable exponent Lebesgue spaces. For $r \in C_+(\close)$ the variable exponent Sobolev space $\Wp{r(\cdot)}$ is defined by
\begin{align*}
	\Wp{r(\cdot)}=\l\{ u \in \Lp{r(\cdot)} \,:\, |\nabla u| \in \Lp{r(\cdot)}\r\}
\end{align*}
endowed with the norm
\begin{align*}
	\|u\|_{1,r(\cdot)}=\|u\|_{r(\cdot)}+\|\nabla u\|_{r(\cdot)},
\end{align*}
where $\|\nabla u\|_{r(\cdot)}= \|\,|\nabla u|\,\|_{r(\cdot)}$.

Moreover, we define
\begin{align*}
	\Wpzero{r(\cdot)}= \overline{C^\infty_0(\Omega)}^{\|\cdot\|_{1,r(\cdot)}}.
\end{align*}
The spaces $\Wp{r(\cdot)}$ and $\Wpzero{r(\cdot)}$ are both separable and 
reflexive Banach spaces, in fact uniformly convex Banach spaces. In the space $\Wpzero{r(\cdot)}$, the Poincar\'e inequality holds, namely
\begin{align*}
	\|u\|_{r(\cdot)} \leq c_0 \|\nabla u\|_{r(\cdot)} \quad\text{for all } u 
\in \Wpzero{r(\cdot)}
\end{align*}
with some $c_0>0$. Therefore, we can consider on $\Wpzero{r(\cdot)}$ the equivalent norm
\begin{align*}
	\|u\|_{1,r(\cdot),0}=\|\nabla u\|_{r(\cdot)} \quad\text{for all } u \in \Wpzero{r(\cdot)}.
\end{align*}
For $r \in C_+(\close)$ we introduce the critical Sobolev variable exponents $r^*$ and $r_*$ defined by
\begin{align*}
	r^*(x)=
	\begin{cases}
		\frac{Nr(x)}{N-r(x)} & \text{if }r(x)<N,\\
		\ell_1(x)& \text{if } N \leq r(x),
	\end{cases} \quad\text{for all }x\in\close
\end{align*}
and
\begin{align*}
	r_*(x)=
	\begin{cases}
		\frac{(N-1)r(x)}{N-r(x)} & \text{if }r(x)<N,\\
		\ell_2(x)& \text{if } N \leq r(x),
	\end{cases} \quad\text{for all }x\in\close,
\end{align*}
where $\ell_1, \ell_2 \in C(\close)$ are arbitrarily chosen such that $r(x)<\ell_1(x)$ for all $x \in \close$ and $r(x)<\ell_2(x)$ for all $x \in \close$.

Furthermore, we denote by $C^{0, \frac{1}{|\log t|}}(\close)$ the set of all functions $h\colon \close \to \R$ that are log-H\"older continuous, that is, there exists $C>0$ such that
\begin{align}\label{log_hoelder}
	|h(x)-h(y)| \leq \frac{C}{|\log |x-y||}\quad\text{for all } x,y\in \close \text{ with } |x-y|<\frac{1}{2}.
\end{align}

Now we can state the embedding from $\Wp{r(\cdot)}$ into $\Lp{r^*(\cdot)}$ under condition \eqref{log_hoelder}, see Diening-Harjulehto-H\"{a}st\"{o}-R$\mathring{\text{u}}$\v{z}i\v{c}ka \cite[Corollary 8.3.2]{Diening-Harjulehto-Hasto-Ruzicka-2011} or Fan \cite[Proposition 2.2]{Fan-2010} and Fan-Shen-Zhao \cite{Fan-Shen-Zhao-2001}.

\begin{proposition}\label{embedding_critical}
	Let $r\in C^{0, \frac{1}{|\log t|}}(\close) \cap C_+(\close)$ and let $s\in C(\close)$ be such that
	\begin{align*}
		1\leq  s(x)\leq r^*(x) \quad \text{for all }x\in\close.
	\end{align*}
	Then, we have the continuous embedding
	\begin{align*}
		W^{1,r(\cdot)}(\Omega) \hookrightarrow L^{s(\cdot) }(\Omega ).
	\end{align*}
	If $r\in C_+(\close)$, $s\in C(\close)$ and $1\leq s(x)< r^*(x)$ for all 
$x\in\overline{\Omega}$, then the embedding above is compact.
\end{proposition}

In the same way we have the embedding into the boundary Lebesgue space, see Fan \cite[Proposition 2.1]{Fan-2010} and Ho-Kim-Winkert-Zhang \cite[Proposition 2.5]{Ho-Kim-Winkert-Zhang-2020} for the continuous and Fan \cite[Corollary 2.4]{Fan-2008} for the compact embedding.

\begin{proposition}\label{embedding_critical_boundary}
	Suppose that $r\in C_+(\close)\cap W^{1,\gamma}(\Omega)$ for some $\gamma>N$. Let $s\in C(\close)$ be such that
	\begin{align*}
		1\leq  s(x)\leq r_*(x) \quad \text{for all }x\in\close.
	\end{align*}
	Then, we have the continuous embedding
	\begin{align*}
		W^{1,r(\cdot)}(\Omega)\hookrightarrow L^{s(\cdot) }(\partial \Omega).
	\end{align*}
	If $r\in C_+(\close)$, $s\in C(\close)$ and $1\leq s(x)< r_*(x)$ for all 
$x\in\overline{\Omega}$, then the embedding above is compact.
\end{proposition}

\begin{remark}
	Note that for a bounded domain $\Omega\subset \R^N$ and $\gamma>N$ we have the following inclusions
	\begin{align*}
		C^{0,1}(\close)\subset \Wp{\gamma}\subset C^{0,1-\frac{N}{\gamma}}(\close) \subset C^{0, \frac{1}{|\log t|}}(\close).
	\end{align*}
\end{remark}

Finally, we recall the relation between the norm and the related modular function which is defined by
\begin{align*}
	\varrho_{r(\cdot)}(u) =\into |u|^{r(x)}\diff x \quad\text{for all } u\in\Lp{r(\cdot)} \text{ with }r\in C_+(\close).
\end{align*}

\begin{proposition}\label{proposition_1}
	If $r\in C_+(\close)$ and $u\in \Lp{r(\cdot)}$, then we have the following assertions:
	\begin{enumerate}
		\item[\textnormal{(i)}]
			$\|u\|_{r(\cdot)}=\lambda \quad\Longleftrightarrow\quad \varrho_{r(\cdot)}\l(\frac{u}{\lambda}\r)=1$ with $u \neq 0$;
		\item[\textnormal{(ii)}]
			$\|u\|_{r(\cdot)}<1$ (resp. $=1$, $>1$) $\quad\Longleftrightarrow\quad \varrho_{r(\cdot)}(u)<1$ (resp. $=1$, $>1$);
		\item[\textnormal{(iii)}]
			$\|u\|_{r(\cdot)}<1$ $\quad\Longrightarrow\quad$ $\|u\|_{r(\cdot)}^{r_+} \leq \varrho_{r(\cdot)}(u) \leq \|u\|_{r(\cdot)}^{r_-}$;
		\item[\textnormal{(iv)}]
			$\|u\|_{r(\cdot)}>1$ $\quad\Longrightarrow\quad$ $\|u\|_{r(\cdot)}^{r_-} \leq \varrho_{r(\cdot)}(u) \leq \|u\|_{r(\cdot)}^{r_+}$;
		\item[\textnormal{(v)}]
			$\|u_n\|_{r(\cdot)} \to 0 \quad\Longleftrightarrow\quad\varrho_{r(\cdot)}(u_n)\to 0$;
		\item[\textnormal{(vi)}]
			$\|u_n\|_{r(\cdot)}\to +\infty \quad\Longleftrightarrow\quad \varrho_{r(\cdot)}(u_n)\to +\infty$.
		\item[\textnormal{(vii)}]
			$\|u_n\|_{r(\cdot)}\to 1 \quad\Longleftrightarrow\quad \varrho_{r(\cdot)}(u_n)\to 1$.
		\item[\textnormal{(viii)}]
			$u_n \to u$ in $\Lp{r(\cdot)} \quad\Longrightarrow\quad  \varrho_{r(\cdot)} (u_n) \to \varrho_{r(\cdot)} (u)$.
	\end{enumerate}
\end{proposition}

Now we recall some definitions and properties concerning Musielak-Orlicz spaces which are mainly taken from the book of Musielak \cite{Musielak-1983}. We also refer to the books of Diening-Harjulehto-H\"{a}st\"{o}-R$\mathring{\text{u}}$\v{z}i\v{c}ka  \cite{Diening-Harjulehto-Hasto-Ruzicka-2011}  and  Harjulehto-H\"{a}st\"{o} \cite{Harjulehto-Hasto-2019} as well as the papers of Colasuonno-Squassina \cite{Colasuonno-Squassina-2016} and 
Fan \cite{Fan-2012}.

We start with the following definition.
\begin{definition}\label{def_phi-function}$~$
	\begin{enumerate}
		\item[\textnormal{(i)}]
			A continuous and convex function $\ph\colon[0,\infty)\to[0,\infty)$ is 
said to be a $\Phi$-function if $\ph(0)=0$ and $\ph(t)>0$ for all $t >0$.
		\item[\textnormal{(ii)}]
			A function $\ph\colon\Omega \times [0,\infty)\to[0,\infty)$ is said to 
be a generalized $\Phi$-function if $\ph(\cdot,t)$ is measurable for all $t\geq 0$ and $\ph(x,\cdot)$ is a $\Phi$-function for a.\,a.\,$x\in\Omega$. We denote the set of all generalized $\Phi$-functions on $\Omega$ by $\Phi(\Omega)$.
		\item[\textnormal{(iii)}]
			A function $\ph\in\Phi(\Omega)$ is locally integrable if $\ph(\cdot,t) 
\in L^{1}(\Omega)$ for all $t>0$.
		\item[\textnormal{(iv)}]
			A function $\ph\in\Phi(\Omega)$ satisfies the $\Delta_2$-condition if there exist a positive constant $C$ and a nonnegative function $h\in\Lp{1}$ such that
			\begin{align*}
				\ph(x,2t) \leq C\ph(x,t)+h(x) 
			\end{align*}
			for a.\,a.\,$x\in\Omega$ and for all $t\in [0,\infty)$.
		\item[\textnormal{(v)}]
			Given $\ph, \psi \in \Phi(\Omega)$, we say that $\ph$ is weaker than $\psi$, denoted by $\ph \prec \psi$, if there exist two positive constants 
$C_1, C_2$ and a nonnegative function $h\in\Lp{1}$ such that
			\begin{align*}
				\ph(x,t) \leq C_1 \psi(x,C_2t)+h(x)
			\end{align*}
			for a.\,a.\,$x\in \Omega$ and for all $t \in[0,\infty)$.
	\end{enumerate}
\end{definition}

For a given $\ph \in \Phi(\Omega)$ we define the corresponding modular $\rho_\ph$ by
\begin{align}\label{modular_orlicz}
	\rho_\ph(u):= \into \ph\l(x,|u|\r)\diff x.
\end{align}
Then, the Musielak-Orlicz space $L^\ph(\Omega)$ is defined by
\begin{align*}
	L^\ph(\Omega):=\left \{u \in M(\Omega)\,:\, \text{there exists }\alpha>0 \text{ such that }\rho_\ph(\alpha u)< +\infty \right \}
\end{align*}
equipped with the norm
\begin{align*}
	\|u\|_{\ph}:=\inf \l\{\alpha >0 \, : \, \rho_\ph \l(\frac{u}{\alpha}\r)\leq 1\r\}.
\end{align*}

The following proposition can be found in Musielak \cite[Theorem 7.7 and Theorem 8.5]{Musielak-1983}.

\begin{proposition}\label{prop_complete}$~$
	\begin{enumerate}
		\item[\textnormal{(i)}]
			Let $\ph \in \Phi(\Omega)$. Then the Musielak-Orlicz space $\Lp{\ph}$ is complete with respect to the norm $\|\cdot\|_\ph$, that is, $\l(\Lp{\ph},\|\cdot\|_\ph\r)$ is a Banach space.
		\item[\textnormal{(ii)}]
			Let $\ph,\psi \in \Phi(\Omega)$ be locally integrable with $\ph \prec \psi$. Then
			\begin{align*}
				\Lp{\psi} \hookrightarrow \Lp{\ph}.
			\end{align*}
	\end{enumerate}
\end{proposition}

Next, we recall the unit ball property, see the books of Musielak \cite[Theorem 8.13]{Musielak-1983} and Diening-Harjulehto-H\"{a}st\"{o}-R$\mathring{\text{u}}$\v{z}i\v{c}ka  \cite[Lemma 2.1.14]{Diening-Harjulehto-Hasto-Ruzicka-2011}.

\begin{proposition}\label{prop_delta_two_and_modular}
	Let $\ph \in \Phi(\Omega)$.
	\begin{enumerate}
		\item[\textnormal{(i)}]
			If $\ph$ satisfy the  $\Delta_2$-condition, then
			\begin{align*}
				L^\ph(\Omega)=\left \{u \in M(\Omega)\,:\,\rho_\ph(u)< +\infty \right \}.
			\end{align*}
		\item[\textnormal{(ii)}]
			Furthermore, if $u \in \Lp{\ph}$, then $\rho_\ph(u)<1$ (resp.\,$=1$; 
$>1)$ if and only if $\|u\|_\ph<1$ (resp.\,$=1$; $>1$).
	\end{enumerate}
\end{proposition}

%
%
%

  Now we are in the position to give the definition of a $N$-function.

\begin{definition}
	The function $\ph\colon [0,\infty) \to [0,\infty)$ is called $N$-function if it is a $\Phi$-function such that
	\begin{align*}
		\lim_{t\to 0^+} \frac{\ph(t)}{t}=0
		\quad\text{and}\quad
		\lim_{t\to\infty} \frac{\ph(t)}{t}=\infty.
	\end{align*}
	We call a function $\ph\colon\Omega\times \R\to[0,\infty)$ a generalized 
$N$-function if $\ph(\cdot,t)$ is measurable for all $t \in \R$ and $\ph(x,\cdot)$ is a $N$-function for a.\,a.\,$x\in\Omega$. We denote the class 
of all generalized $N$-functions by $N(\Omega)$. Note that $\ph^* \in N(\Omega)$ whenever $\ph \in N(\Omega)$.
\end{definition}

\begin{definition}
	Let $\phi, \psi \in N(\Phi)$. The function $\phi$ increases essentially slower than $\psi$ near infinity, if for any $k>0$
	\begin{align*}
		\lim_{t\to\infty} \frac{\phi(x,kt)}{\psi(x,t)}=0 \quad \text{uniformly for a.\,a.\,}x\in\Omega.
	\end{align*}
	We write $\phi \ll \psi$.
\end{definition}

Let $\ph\in\Phi(\Omega)$. The corresponding Sobolev space $\Wp{\ph}$ is defined by
\begin{align*}
	\Wp{\ph} := \left \{u \in \Lp{\ph} \,:\, |\nabla u| \in \Lp{\ph} \right\}
\end{align*}
equipped with the norm
\begin{align*}
	\|u\|_{1,\ph} = \|u\|_\ph+\|\nabla u\|_\ph
\end{align*}
where $\|\nabla u\|_\ph=\| \, |\nabla u| \,\|_\ph$. We denote $\rho_\ph ( \nabla u ) = \rho_\ph ( |\nabla u| )$ as well. If $\ph \in N(\Omega)$ is locally integrable, we denote by $\Wpzero{\ph}$ the completion of $C^\infty_0(\Omega)$ in $\Wp{\ph}$.

The next theorem gives a criterion when the Sobolev spaces are Banach spaces and also reflexive. This result can be found in Musielak \cite[Theorem 10.2]{Musielak-1983} and  Fan \cite[Proposition 1.7 and 1.8]{Fan-2012b}.

\begin{theorem}\label{prop_complete2}
	Let $\ph\in N(\Omega)$ be locally integrable such that
	\begin{align}\label{condition_ph}
		\inf_{x\in\Omega} \ph(x,1)>0.
	\end{align}
	Then the spaces $\Wp{\ph}$ and $\Wpzero{\ph}$ are separable Banach spaces which are reflexive if $\Lp{\ph}$ is reflexive.
\end{theorem}

Let us now come to our special double phase $N$-function. To this end, let $\mathcal{H} \colon \Omega \times [0,\infty) \to [0,\infty)$ be defined 
as
\begin{align*}
	\mathcal{H}(x,t):=t^{p(x)} +\mu(x)t^{q(x)} \quad\text{for all } (x,t)\in \Omega \times [0,\infty),
\end{align*}
where we suppose the following:
\begin{align}\label{condition_1}
	\begin{split}
		&\Omega\subseteq \R^N, N\geq 2, \text{ is a bounded domain with Lipschitz boundary }\partial\Omega,\\
		&p,q\in C_+(\close) \text{ such that }1<p(x)<N \text{ and } p(x) < q(x) 
\text{ for all }x\in\close, \\
		&\text{and }0 \leq \mu(\cdot) \in \Lp{1}.
	\end{split}
\end{align}

It is clear that $ \mathcal{H}$ is a locally integrable, generalized $N$-function which satisfies \eqref{condition_ph} and it fulfills the $\Delta_2$-condition, that is,
\begin{align}\label{Delta2-condition}
	\mathcal{H}(x,2t)=(2t)^{p(x)} +\mu(x)(2t)^{q(x)} \leq 2^{q_+} \mathcal{H}(x,t).
\end{align}

Recall that the corresponding modular to $\mathcal{H}$ is given by
\begin{align*}
	\rho_{\mathcal{H}}(u) = \into \mathcal{H} (x,|u|)\diff x.
\end{align*}
Then, the corresponding Musielak-Orlicz space $\Lp{\mathcal{H}}$ is given 
by
\begin{align*}
	L^{\mathcal{H}}(\Omega)=\left \{u \in M(\Omega) \,:\,\rho_{\mathcal{H}}(u) < +\infty \right \},
\end{align*}
see Proposition \ref{prop_delta_two_and_modular}, endowed with the norm
\begin{align*}
	\|u\|_{\mathcal{H}} = \inf \left \{ \tau >0 : \rho_{\mathcal{H}}\left(\frac{u}{\tau}\right) \leq 1  \right \}.
\end{align*}

Similarly, we can introduce the spaces $\Wp{\mathcal{H}}$ and $\Wpzero{\mathcal{H}}$ equipped with the norm
\begin{align*}
	\|u\|_{1,\mathcal{H}} = \|u\|_{\mathcal{H}}+\|\nabla u\|_{\mathcal{H}}.
\end{align*}

We recall the following definition which is needed for the reflexivity of 
the spaces $\Lp{\mathcal{H}}$, $\Wp{\mathcal{H}}$ and $\Wpzero{\mathcal{H}}$.

\begin{definition} \label{definition-uniform-convexity}
	A function $\ph \in N(\Omega)$ is said to be uniformly convex if for every $\eps>0$ there exists $\delta>0$ such that
	\begin{align*}
		|t-s| \leq \eps \max \{t,s\}
		\quad \text{or}\quad
		\ph\l(x,\frac{t+s}{2}\r) \leq (1-\delta) \frac{\ph(x,t)+\ph(x,s)}{2}
	\end{align*}
	for all $t,s \geq 0$ and for a.\,a.\,$x\in\Omega$.
\end{definition}

Now we can state the following result which is inspired by the work of Colasuonno-Squassina \cite{Colasuonno-Squassina-2016}.

\begin{proposition} \label{proposition-reflexivity}
	Let hypotheses \eqref{condition_1} be satisfied. Then, the norm $\|\cdot\|_{\mathcal{H}}$ defined on $\Lp{\mathcal{H}}$ is uniformly convex and hence the spaces $\Lp{\mathcal{H}}$, $\Wp{\mathcal{H}}$ and $\Wpzero{\mathcal{H}}$ are reflexive Banach spaces. Furthermore, for any sequence $\{u_n\}_{n \in \N} \subseteq \Lp{\mathcal{H}}$ such that
	\begin{equation*}
		u_n \weak u \text{ in }\Lp{\mathcal{H}}\quad\text{and}\quad
		\rho_{\mathcal{H}} (u_n) \to \rho_{\mathcal{H}} (u)
	\end{equation*}
	it follows that $u_n \to u$ in $\Lp{\mathcal{H}}$.
\end{proposition}

\begin{proof}
	First note that by Propositions \ref{prop_complete} and \ref{prop_complete2} we know that $\Lp{\mathcal{H}}$, $\Wp{\mathcal{H}}$ and $\Wpzero{\mathcal{H}}$ are complete. For the first part we only need to show that $\Lp{\mathcal{H}}$ is uniformly convex, then $\Lp{\mathcal{H}}$ is reflexive by the Milman-Pettis theorem, see, for example, Papageorgiou-Winkert \cite[Theorem 3.4.28]{Papageorgiou-Winkert-2018}. Applying Proposition \ref{prop_complete2} then shows that $\Wp{\mathcal{H}}$ and $\Wpzero{\mathcal{H}}$ are reflexive as well. 
	
	In order to prove the uniform convexity of the space $\Lp{\mathcal{H}}$,
	it is enough to show that the $N$-function $\mathcal{H}$ is uniformly convex, see Diening-Harjulehto-H\"{a}st\"{o}-R$\mathring{\text{u}}$\v{z}i\v{c}ka \cite[Definition 2.4.5, Theorems 2.4.11 and 2.4.14]{Diening-Harjulehto-Hasto-Ruzicka-2011}. Furthermore, the second assertion also follows by the results in Diening-Harjulehto-H\"{a}st\"{o}-R$\mathring{\text{u}}$\v{z}i\v{c}ka \cite[Lemma 2.4.17 and Remark 2.4.19]{Diening-Harjulehto-Hasto-Ruzicka-2011}
	
	To this end, let $\eps>0$ and let $t,s\ge 0$ be such that $|t-s|>\eps\max\{t,s\}$.
	Since the function $t\mapsto t^{\ell}$ is uniformly convex whenever $\ell>1$, see Diening-Harjulehto-H\"{a}st\"{o}-R$\mathring{\text{u}}$\v{z}i\v{c}ka \cite[Remark 2.4.6]{Diening-Harjulehto-Hasto-Ruzicka-2011},
	there exists $\delta_p=\delta_p(\eps,p_-)>0$ such that
	\begin{align*}
		\l(\frac{t+s}{2}\r)^{p_-}\le (1-\delta_p)\frac{t^{p_-}+s^{p_-}}{2}.
	\end{align*}
	Thus, using also the convexity of $t\longmapsto t^{\frac{p(x)}{p_-}}$ for $x\in\Omega$, we get
	\begin{align*}
		\l(\frac{t+s}{2}\r)^{p(x)}\le \l((1-\delta_p)\frac{t^{p_-}+s^{p_-}}{2}\r)^{\frac{p(x)}{p_-}}
		\le (1-\delta_p)\frac{t^{p(x)}+s^{p(x)}}{2}.
	\end{align*}
	Analogously we have
	\begin{align*}
		\l(\frac{t+s}{2}\r)^{q(x)}
		\le (1-\delta_q)\frac{t^{q(x)}+s^{q(x)}}{2}
	\end{align*}
	for some $\delta_q=\delta_q(\eps,q_-)>0$. Finally, we reach
	\begin{align*}
		&\l(\frac{t+s}{2}\r)^{p(x)}+\mu(x)\l(\frac{t+s}{2}\r)^{q(x)}\\
		&\le (1-\min\{\delta_p,\delta_q\})\frac{t^{p(x)}+\mu(x)t^{q(x)}+s^{p(x)}+\mu(x)s^{q(x)}}{2}.
	\end{align*}
	This completes the proof of the uniform convexity of $\Lp{\mathcal{H}}$. 
\end{proof}

Next, we want to check the relation between the modular $\rho_{\mathcal{H}}$ and its norm $\|\cdot\|_\mathcal{H}$, see also Harjulehto-H\"{a}st\"{o} \cite{Harjulehto-Hasto-2019}.

\begin{proposition}\label{proposition_modular_properties}
	Let hypotheses \eqref{condition_1} be satisfied and let $\rho_{\mathcal{H}}$ be defined by
	\begin{align*}
		\rho_{\mathcal{H}}(u) = \into \l(|u|^{p(x)}+\mu(x)|u|^{q(x)}\r)\diff x\quad\text{for all } u \in \Lp{\mathcal{H}}.
	\end{align*}
	\begin{enumerate}
		\item[\textnormal{(i)}]
			If $u\neq 0$, then $\|u\|_{\mathcal{H}}=\lambda$ if and only if $ \rho_{\mathcal{H}}(\frac{u}{\lambda})=1$;
		\item[\textnormal{(ii)}]
			$\|u\|_{\mathcal{H}}<1$ (resp.\,$>1$, $=1$) if and only if $ \rho_{\mathcal{H}}(u)<1$ (resp.\,$>1$, $=1$);
		\item[\textnormal{(iii)}]
			If $\|u\|_{\mathcal{H}}<1$, then $\|u\|_{\mathcal{H}}^{q_+}\leqslant \rho_{\mathcal{H}}(u)\leqslant\|u\|_{\mathcal{H}}^{p_-}$;
		\item[\textnormal{(iv)}]
			If $\|u\|_{\mathcal{H}}>1$, then $\|u\|_{\mathcal{H}}^{p_-}\leqslant \rho_{\mathcal{H}}(u)\leqslant\|u\|_{\mathcal{H}}^{q_+}$;
		\item[\textnormal{(v)}]
			$\|u\|_{\mathcal{H}}\to 0$ if and only if $ \rho_{\mathcal{H}}(u)\to 0$;
		\item[\textnormal{(vi)}]
			$\|u\|_{\mathcal{H}}\to +\infty$ if and only if $ \rho_{\mathcal{H}}(u)\to +\infty$.
		\item[\textnormal{(vii)}]
			$\|u\|_{\mathcal{H}}\to 1$ if and only if $ \rho_{\mathcal{H}}(u)\to 1$.
		\item[\textnormal{(viii)}]
			If $u_n \to u$ in $\Lp{\mathcal{H}}$, then $\rho_{\mathcal{H}} (u_n) \to \rho_{\mathcal{H}} (u)$.
	\end{enumerate}
\end{proposition}

\begin{proof}
	\textnormal{(i)} First note that, for $u \in \Lp{\mathcal{H}}$, the function $\rho_{\mathcal{H}}(\lambda u)$ is continuous, convex and even in the variable $\lambda$ and it is strictly increasing when $\lambda \in [0,+\infty)$. So, by definition, we directly obtain
	\begin{align*}
		\|u\|_{\mathcal{H}}=\lambda \quad\text{if and only if}\quad \rho_{\mathcal{H}}\l(\frac{u}{\lambda}\r)=1,
	\end{align*}
	which proves \textnormal{(i)} and \textnormal{(ii)} follows from \textnormal{(i)}. Let us show \textnormal{(iii)}. For $u \in \Lp{\mathcal{H}}$ we have the inequalities
	\begin{align}\label{est_rho}
		\begin{split}
			& b^{p_-}\rho_{\mathcal{H}}(u) \leq \rho_{\mathcal{H}}(bu)\leq b^{q_+}\rho_{\mathcal{H}}(u) \quad\text{if }b>1,\\
			& b^{q_+}\rho_{\mathcal{H}}(u) \leq \rho_{\mathcal{H}}(bu)\leq b^{p_-}\rho_{\mathcal{H}}(u) \quad\text{if }0<b<1.
		\end{split}
	\end{align}
	Let $\|u\|_{\mathcal{H}}=\lambda$ with $0<\lambda<1$. Then, we have $\rho_{\mathcal{H}}\l(\frac{u}{\lambda}\r)=1$ from \textnormal{(i)}. Since $\frac{1}{\lambda}>1$ we can apply the first inequality in \eqref{est_rho} in order to obtain
	\begin{align*}
		\frac{\rho_{\mathcal{H}}(u)}{\lambda^{p_-}} \leq \rho_{\mathcal{H}}\l(\frac{u}{\lambda}\r)=1\leq
		\frac{\rho_{\mathcal{H}}(u)}{\lambda^{q_+}}.
	\end{align*}
	This shows \textnormal{(iii)}. The same argument can be used in order to  show \textnormal{(iv)} by using the second inequality in \eqref{est_rho}. Moreover, \textnormal{(v)} follows from \textnormal{(iii)}, \textnormal{(vi)} follows from \textnormal{(iv)} and \textnormal{(vii)} follows from \textnormal{(iii)} and \textnormal{(iv)}. Finally, when $u_n \to u$ in $\Lp{\mathcal{H}}$, by \textnormal{(v)} and as both addends are positive, it follows that $\varrho_{p(\cdot)} (u_n - u) \to 0$, hence by Proposition \ref{proposition_1} and the usual embeddings $\| u_n - u \|_{p_-} \to 0$, so $u_n \to u$ a.\,e.\,through a subsequence (still denoted by $u_n$). On the other hand, as
\begin{equation*}
	|u_n|^{p(x)}+\mu(x)|u_n|^{q(x)} 
	\leq 2^{q_+} \left( |u_n-u|^{p(x)} + |u|^{p(x)} + \mu(x) |u_n-u|^{q(x)} + \mu(x) |u|^{q(x)} \right)
\end{equation*}
and by \textnormal{(v)} there holds $\rho_{\mathcal{H}}(u_n - u)\to 0$, we know that $\left\lbrace |u_n|^{p(x)}+\mu(x)|u_n|^{q(x)} \right\rbrace_{n \in \N} $ is a uniformly integrable sequence, which furthermore converges a.\,e.\,to $|u|^{p(x)}+\mu(x)|u|^{q(x)}$ by the a.\,e.\,convergence of $u_n \to u$. By Vitali's Theorem (see Bogachev \cite[Theorem 4.5.4]{Bogachev-2007}) it follows that $\rho_{\mathcal{H}} (u_n) \to \rho_{\mathcal{H}} (u)$ through this subsequence. One can recover the whole sequence by the subsequence principle and this proves \textnormal{(viii)}.
\end{proof}

We now equip the space $\Wp{\mathcal{H}}$ with the equivalent norm
\begin{align*}
	\|u\|_{\hat{\rho}_\mathcal{H}}:=\inf&\left\{\lambda >0 :
	\int_\Omega
	\left[\l|\frac{\nabla u}{\lambda}\r|^{p(x)}+\mu(x)\l|\frac{\nabla u}{\lambda}\r|^{q(x)}+\l|\frac{u}{\lambda}\r|^{p(x)}+\mu(x)\l|\frac{u}{\lambda}\r|^{q(x)}\right]\;\diff x\le1\right\},
\end{align*}
where the modular $\hat{\rho}_\mathcal{H}$ is given by
\begin{align}\label{modular2}
	\hat{\rho}_\mathcal{H}(u) =\into \l(|\nabla u|^{p(x)}+\mu(x)|\nabla u|^{q(x)}\r)\diff x+\into \l(|u|^{p(x)}+\mu(x)|u|^{q(x)}\r)\diff x
\end{align}
for $u \in \Wp{\mathcal{H}}$.

The following proposition gives the relation between the norm $\|\cdot\|_{\hat{\rho}_\mathcal{H}}$ and the corresponding modular function $\hat{\rho}_\mathcal{H}$. The proof is similar to that one of Proposition \ref{proposition_modular_properties}.

\begin{proposition}\label{proposition_modular_properties2}
	Let hypotheses \eqref{condition_1} be satisfied, let $y\in\Wp{\mathcal{H}}$ and let $\hat{\rho}_{\mathcal{H}}$ be defined as in  \eqref{modular2}.
	\begin{enumerate}
		\item[\textnormal{(i)}]
			If $y\neq 0$, then $\|y\|_{\hat{\rho}_\mathcal{H}}=\lambda$ if and only if $ \hat{\rho}_{\mathcal{H}}(\frac{y}{\lambda})=1$;
		\item[\textnormal{(ii)}]
			$\|y\|_{\hat{\rho}_\mathcal{H}}<1$ (resp.\,$>1$, $=1$) if and only if $ \hat{\rho}_{\mathcal{H}}(y)<1$ (resp.\,$>1$, $=1$);
		\item[\textnormal{(iii)}]
			If $\|y\|_{\hat{\rho}_\mathcal{H}}<1$, then $\|y\|_{\hat{\rho}_\mathcal{H}}^{q_+}\leqslant \hat{\rho}_{\mathcal{H}}(y)\leqslant\|y\|_{\hat{\rho}_\mathcal{H}}^{p_-}$;
		\item[\textnormal{(iv)}]
			If $\|y\|_{\hat{\rho}_\mathcal{H}}>1$, then $\|y\|_{\hat{\rho}_\mathcal{H}}^{p_-}\leqslant \hat{\rho}_{\mathcal{H}}(y)\leqslant\|y\|_{\hat{\rho}_\mathcal{H}}^{q_+}$;
		\item[\textnormal{(v)}]
			$\|y\|_{\hat{\rho}_\mathcal{H}}\to 0$ if and only if $ \hat{\rho}_{\mathcal{H}}(y)\to 0$;
		\item[\textnormal{(vi)}]
			$\|y\|_{\hat{\rho}_\mathcal{H}}\to +\infty$ if and only if $ \hat{\rho}_{\mathcal{H}}(y)\to +\infty$.
		\item[\textnormal{(vii)}]
			$\|y\|_{\hat{\rho}_\mathcal{H}}\to 1$ if and only if $ \hat{\rho}_{\mathcal{H}}(y)\to 1$.
		\item[\textnormal{(viii)}]
			If $u_n \to u$ in $\Wp{\mathcal{H}}$, then $\hat{\rho}_\mathcal{H} (u_n) \to \hat{\rho}_\mathcal{H} (u)$.
	\end{enumerate}
\end{proposition}

Moreover, this norm is a uniformly convex norm on this space and satisfies the Radon–Riesz (or Kadec-Klee) property with respect to the modular, as one can see in the following proposition. 

\begin{proposition} \label{proposition-W1H-unifconv}
	Let hypotheses \eqref{condition_1} be satisfied.
	\begin{enumerate}
		\item[\textnormal{(i)}]	The norm $\|  \cdot \|_{ \hat{\rho}_{\mathcal{H}} }$ on $\Wp{\mathcal{H}}$ is uniformly convex. 
		\item[\textnormal{(ii)}]	For any sequence $\{u_n\}_{n \in \N} \subseteq \Wp{\mathcal{H}}$ such that 
		\begin{equation*}
			u_n \weak u \text{ in }\Wp{\mathcal{H}}\quad\text{and}\quad
			\hat{\rho}_{\mathcal{H}} (u_n) \to \hat{\rho}_{\mathcal{H}} (u)
		\end{equation*}
		it holds that $u_n \to u$ in $\Wp{\mathcal{H}}$.
	\end{enumerate}
\end{proposition}

\begin{proof}
	Both results follow by Theorems 3.2, 3.5 of Fan-Guan \cite{Fan-Guan-2001}. First note that their condition $\Delta_{2,\delta(x)}$ is a more general version of our $\Delta_2$-condition. Moreover (UC)$_1$ is exactly the same as the uniform convexity from Definition \ref{definition-uniform-convexity}, by taking $t = \beta s$ with $s\geq t$, the left-hand side condition from Definition \ref{definition-uniform-convexity} is equivalent to $\beta < 1 - \eps$ for $\eps < 1$, hence one can take $\sigma(1-\eps) = \delta$. In the proof of Proposition \ref{proposition-reflexivity} we already verified these properties for $\mathcal{H}$, so \textnormal{(i)} follows by Theorem 3.2 of Fan-Guan \cite{Fan-Guan-2001}. Regarding the assumptions (Q), take $Y = \Wp{\mathcal{H}} = X$ and the modular $\hat{\rho}_{\mathcal{H}}$, then we already know that (Q$_1$)-(Q$_7$) hold in Proposition \ref{proposition_modular_properties2}. Fix now any $c > 0$. Then the $N$-function given by $c\mathcal{H}$ is uniformly convex and $\Wp{c\mathcal{H}} = \Wp{\mathcal{H}}$ as sets, so again by Theorem 3.2 of Fan-Guan \cite{Fan-Guan-2001} the norm $\| \cdot \|_{\hat{\rho}_{c\mathcal{H}}}$ on $\Wp{\mathcal{H}}$ is uniformly convex. Note also that $\hat{\rho}_{\mathcal{H}} (u) \to 0$ if and only if $\hat{\rho}_{c\mathcal{H}} (u) \to 0$, so they generate the same topology. It is straightforward that (Q$_1$) and (Q$_2$) hold for $c\mathcal{H}$, and (Q$_3$)-(Q$_7$) follow doing an analogous argument to Proposition \ref{proposition_modular_properties2}. Thus \textnormal{(ii)} follows from Theorem 3.5 of Fan-Guan \cite{Fan-Guan-2001}. 
\end{proof}

Now we introduce the seminormed space
\begin{align*}
	L^{q(\cdot)}_\mu(\Omega)=\left \{u\in M(\Omega) \,:\, \into \mu(x) | u|^{q(x)} \diff x< +\infty \right \}
\end{align*}
and endow it with the seminorm
\begin{align*}
	\|u\|_{q(\cdot),\mu} =\inf \left \{ \tau >0 \,:\, \into \mu(x) \l(\frac{|u|}{\tau}\r)^{q(x)} \diff x  \leq 1  \right \}.
\end{align*}

We have the following embedding results, see Proposition 2.15 of Colasuonno-Squassina \cite{Colasuonno-Squassina-2016} for the constant exponent case.

\begin{proposition}\label{proposition_embeddings}
	Let hypotheses \eqref{condition_1} be satisfied and let
	\begin{align}\label{critical_exponent}
		p^*(x):=\frac{Np(x)}{N-p(x)} \quad \text{and}\quad p_*(x):=\frac{(N-1)p(x)}{N-p(x)}
	\quad\text{for all }x\in\close
	\end{align}
	be the critical exponents to $p$. Then the following embeddings hold:
	\begin{enumerate}
		\item[\textnormal{(i)}]
			$\Lp{\mathcal{H}} \hookrightarrow \Lp{r(\cdot)}$, $\Wp{\mathcal{H}}\hookrightarrow \Wp{r(\cdot)}$, $\Wpzero{\mathcal{H}}\hookrightarrow \Wpzero{r(\cdot)}$
            are continuous for all $r\in C(\close)$ with $1\leq r(x)\leq p(x)$ for all $x \in \Omega$;
		\item[\textnormal{(ii)}]
			if $p \in C_+(\close) \cap C^{0, \frac{1}{|\log t|}}(\close)$, then $\Wp{\mathcal{H}} \hookrightarrow \Lp{r(\cdot)}$ and $\Wpzero{\mathcal{H}} \hookrightarrow \Lp{r(\cdot)}$ are continuous for $r \in C(\close)$ with $ 1 \leq r(x) \leq p^*(x)$ for all $x\in \close$;
		\item[\textnormal{(iii)}]
			$\Wp{\mathcal{H}} \hookrightarrow \Lp{r(\cdot)}$ and $\Wpzero{\mathcal{H}} \hookrightarrow \Lp{r(\cdot)}$ are compact for $r \in C(\close) $ with $ 1 \leq r(x) < p^*(x)$ for all $x\in \close$;
		\item[\textnormal{(iv)}]
			if $p \in C_+(\close) \cap W^{1,\gamma}(\Omega)$ for some $\gamma>N$, then $\Wp{\mathcal{H}} \hookrightarrow \Lprand{r(\cdot)}$ and $\Wpzero{\mathcal{H}} \hookrightarrow \Lprand{r(\cdot)}$ are continuous for $r \in C(\close)$ with $ 1 \leq r(x) \leq p_*(x)$ for all $x\in \close$;
		\item[\textnormal{(v)}]
			$\Wp{\mathcal{H}} \hookrightarrow \Lprand{r(\cdot)}$ and $\Wpzero{\mathcal{H}} \hookrightarrow \Lprand{r(\cdot)}$ are compact for $r \in C(\close) $ with $ 1 \leq r(x) < p_*(x)$ for all $x\in \close$;
		\item[\textnormal{(vi)}]
			$\Lp{\mathcal{H}} \hookrightarrow L^{q(\cdot)}_\mu(\Omega)$ is continuous;
		\item[\textnormal{(vii)}]
			if $\mu \in \Linf$, then $L^{q(\cdot)}(\Omega) \hookrightarrow \Lp{\mathcal{H}}$ is continuous.
	\end{enumerate}
\end{proposition}

\begin{proof}
	We take $\mathcal{H}_{p(\cdot)}(x,t)=t^{p(x)}$ for all $t\geq 0$ and for all $x\in\close$. It is easy to see that $\mathcal{H}_{p(\cdot)}  \prec \mathcal{H}$, see Definition \ref{def_phi-function} \textnormal{(v)}. Hence, from Proposition \ref{prop_complete} \textnormal{(ii)} we obtain that $\Lp{\mathcal{H}} \hookrightarrow \Lp{p(\cdot)}$ and $\Wp{\mathcal{H}} 
\hookrightarrow \Wp{p(\cdot)}$ continuously, and by definition it follows 
that $\Wpzero{\mathcal{H}} \hookrightarrow \Wpzero{p(\cdot)}$ continuously. Thus, assertion \textnormal{(i)} is a direct consequence of the classical embedding results for variable Lebesgue and Sobolev spaces due to the 
boundedness of $\Omega$. The same arguments show \textnormal{(ii)--(v)}, see also Propositions \ref{embedding_critical} and \ref{embedding_critical_boundary}. Let us prove \textnormal{(vi)}. To this end, let $u \in \Lp{\mathcal{H}}$, then we have
	\begin{align*}
		\into \mu(x) |u|^{q(x)}\diff x \leq \into \l(|u|^{p(x)}+\mu(x)|u|^{q(x)}\r)\diff x = \rho_{\mathcal{H}}(u),
	\end{align*}
	see \eqref{modular_orlicz}. Since $\rho_{\mathcal{H}}\l(\frac{u}{\|u\|_\mathcal{H}}\r)=1$ whenever $u\neq 0$, we obtain for $u \neq 0$
	\begin{align*}
		\into \mu(x) \l(\frac{u}{\|u\|_\mathcal{H}}\r)^{q(x)}\diff x \leq 1.
	\end{align*}
	Thus
	\begin{align*}
		\|u\|_{q(\cdot),\mu} \leq \|u\|_\mathcal{H}.
	\end{align*}
	Finally, assertion \textnormal{(vii)} follows from the estimate
	\begin{align*}
		\mathcal{H}(x,t)\leq \l(1+t^{q(x)}\r) + \mu(x) t^{q(x)} \leq 1+ \l(1+\|\mu\|_\infty\r) t^{q(x)}
	\end{align*}
	for all $t \geq 0$ and for a.\,a.\,$x \in \Omega$ by applying again Proposition \ref{prop_complete} \textnormal{(ii)}.
\end{proof}

A useful property for existence results is the fact that a space is closed with respect to truncations. We prove this property for $\Wp{\mathcal{H}}$ and $\Wpzero{\mathcal{H}}$ in the following proposition.
For any $s \in \R$ we denote $s^\pm = \max \{ \pm s, 0 \}$, that means $s = s^+ - s^-$ and $|s| = s^+ + s^-$. For any function $v \colon \Omega \to \R$ we denote $v^\pm (\cdot) = [v(\cdot)]^\pm$.

\begin{proposition}
	Let \eqref{condition_1} be satisfied, then the following hold:
	\begin{enumerate}
		\item[\textnormal{(i)}] if $u \in \Wp{\mathcal{H}}$, then $\pm u^\pm \in \Wp{\mathcal{H}}$ with $\nabla (\pm u ^\pm) = \nabla u 1_{\{\pm u > 0\}}$;
		\item[\textnormal{(ii)}] if $u_n \to u$ in $\Wp{\mathcal{H}}$, then $\pm u_n^\pm \to \pm u^\pm$ in $\Wp{\mathcal{H}}$;
		\item[\textnormal{(iii)}] if $\mu \in \Lp{\infty}$ and $u \in \Wpzero{\mathcal{H}}$, then $\pm u^\pm \in \Wpzero{\mathcal{H}}$.
	\end{enumerate} 
\end{proposition}

\begin{proof}
	\textnormal{(i)} It is a well-known fact that for $v \in \Wp{\ell}$, where $1 \leq \ell \leq \infty$, the statement holds and $\nabla (\pm v ^\pm) = \nabla v 1_{\{\pm v > 0\}}$, see for example Heinonen-Kilpeläinen-Martio \cite[Lemma 1.19]{HeinonenKilpelainenMartio2006}. Hence, by Proposition \ref{proposition_embeddings} \textnormal{(i)}, $u \in \Wp{p_-}$ and $\nabla (\pm u ^\pm) = \nabla u 1_{\{\pm u > 0\}}$ it follows that
	\begin{align*}
		\rho_{\mathcal{H}} (\pm u^\pm ) &\leq \rho_{\mathcal{H}} ( u ) < \infty,\\
		\rho_{\mathcal{H}} (\nabla [\pm u^\pm] ) &\leq \rho_{\mathcal{H}} ( \nabla u ) < \infty.
	\end{align*}
	
	\textnormal{(ii)} Consider a sequence such that $u_n \to u$ in $\Wp{\mathcal{H}}$. As $| \pm u_n^\pm \mp u^\pm | \leq | u_n - u | $ pointwisely in $\Omega$, it is straightforward that $\pm u_n ^\pm \to \pm u^\pm$ in $\Lp{\mathcal{H}}$ by proving the convergence in $\rho_{\mathcal{H}}$ and Proposition \ref{proposition_modular_properties} \textnormal{(v)}. For the convergence of the gradients, consider  
	\begin{align*}
		& \into |\pm \nabla u_n^\pm \mp \nabla u^\pm|^{p(x)} \diff x\\
		&= \into | 1_{\{ \pm u_n > 0 \}} 			\nabla u_n - 1_{\{ \pm u > 0 \}} 			\nabla u|^{p(x)} \diff x \\
		& \leq 2^{p_+} \into | \nabla u_n - \nabla u|^{p(x)} \diff x 
		+ 2^{p_+} \into | \nabla u|^{p(x)} |  1_{\{ \pm u_n > 0 \}} - 1_{\{ \pm u_n > 0 \}} |^{p(x)}  \diff x,
	\end{align*}
	where the first term converges to zero by Proposition \ref{proposition_embeddings} \textnormal{(i)} and Proposition \ref{proposition_modular_properties} \textnormal{(v)}, and the second term converges to zero by taking a.e. convergent subsequences and using the Dominated Convergence Theorem (note that $\nabla u = 0$ a.\,e.\,on the set $\{u=0\}$ because, by \textnormal{(i)}, $\nabla u = \nabla u 1_{\{u > 0\}} + \nabla u 1_{\{u < 0\}}$), and then make use of the subsequence principle. Using exactly the same argument we can prove the analogous convergence with exponent $q(x)$ and weight $\mu(x)$, i.e., there holds
	\begin{equation*}
		\rho_{\mathcal{H}} \left( \pm \nabla u_n ^\pm \mp \nabla u ^\mp \right) \to 0 \quad \text{ as } n \to \infty,
	\end{equation*}
	which by Proposition \ref{proposition_modular_properties} \textnormal{(v)} implies $ \pm \nabla u_n^\pm \to \pm \nabla u ^\pm$ in $\Lp{\mathcal{H}}$ and the proof is complete.
	
	\textnormal{(iii)} By definition of $\Wpzero{\mathcal{H}}$, there exists a sequence $\{ v_n \}_{n \in \N} \subseteq C^\infty_0(\Omega)$ such that $v_n \to u$ in $\Wp{\mathcal{H}}$. By \textnormal{(ii)}, we know that $\pm v_n^\pm \to \pm u^\pm$ in $\Wp{\mathcal{H}}$.

	Note that $\{ \pm v_n^\pm \}_{n \in \N} \subseteq C_0(\Omega) = \{ v \in C(\Omega) : \supp v \text{ is compact} \}$ and, by \textnormal{(i)}, their weak derivatives satisfy $\{ \pm \partial_{x_i} v_n^\pm \}_{n \in \N} \subseteq \Lp{\infty}$. By using the standard mollifier $\eta_\eps$, for $\eps$ small enough there holds $\{ \eta_\eps \ast (\pm v_n^\pm) \}_{n \in \N} \subseteq C^\infty_0(\Omega)$. Furthermore, we also have the convergences 
	\begin{align*}
		& \eta_\eps \ast (\pm v_n^\pm) \to \pm v_n^\pm \quad \text{ uniformly in $\Omega$ as } \eps \to 0,\\
		& \partial_{x_i} (\eta_\eps \ast (\pm  v_n^\pm)) = \eta_\eps \ast (\pm \partial_{x_i} v_n^\pm) \to \pm \partial_{x_i} v_n^\pm \quad \text{ in } \Lp{q_+} \text{ as } \eps \to 0.
	\end{align*}
	Hence, $\eta_\eps \ast (\pm v_n^\pm) \to \pm v_n^\pm$ in $\Wp{\mathcal{H}}$ by Proposition \ref{proposition_modular_properties} \textnormal{(v)} via checking the convergence in $\rho_{\mathcal{H}}$ (note that $\Lp{q_+} \hookrightarrow \Lp{q(\cdot)} \hookrightarrow \Lp{\mathcal{H}}$ by Proposition \ref{proposition_embeddings} \textnormal{(vii)}).
	This means that for each $\pm v_n^\pm$ we can find another $\tilde{v}_{\pm,n} \in C_0^\infty(\Omega)$ as close to $\pm v_n^\pm$ as we want in the norm $\| \cdot \|_{\mathcal{H}}$ and this new sequence satisfies $\tilde{v}_{\pm,n} \to \pm u^\pm$ in $\Wp{\mathcal{H}}$.
	
\end{proof}

From Proposition \ref{proposition_embeddings} we can derive the compact embedding of $\Wp{\mathcal{H}}$ into $\Lp{\mathcal{H}}$ and a equivalent norm for $\Wpzero{\mathcal{H}}$. For this purpose we need the following assumptions, more restrictive as in \eqref{condition_1}.
\begin{enumerate}
	\item[\textnormal{(H)}]
		$\Omega\subseteq \R^N$, $N\geq 2$, is a bounded domain with Lipschitz boundary $\partial\Omega$, $0\leq \mu(\cdot)\in \Lp{\infty}$ and $p, q \in 
C_+(\close)$ are such that 
		\begin{enumerate}
			\item[\textnormal{(i)}] 
				$1 < p(x) < N$ for all $x \in \close$;
			\item[\textnormal{(ii)}] 
				$p(x) < q(x) < p^*(x)$ for all $x \in \close$.
	\end{enumerate}
\end{enumerate}
Note that in any case they are significantly less restrictive as those used for the same purpose in Proposition 2.18 of Colasuonno-Squassina \cite{Colasuonno-Squassina-2016}. For many results, this is the only reason to ask for so restrictive assumptions, so they could be generalized to \textnormal{(H)}.

\begin{proposition} \label{prop-generalpoincare}
	Let hypothesis \textnormal{(H)} be satisfied. Then the following hold:
	\begin{enumerate}
		\item[\textnormal{(i)}]
		$\Wp{\mathcal{H}}\hookrightarrow \Lp{\mathcal{H}}$ is a compact embedding;
		\item[\textnormal{(ii)}]
		There exists a constant $C>0$ independent of $u$ such that 
		\begin{align*}
			\|u\|_{\mathcal{H}} \leq C\|\nabla u\|_{\mathcal{H}}\quad\text{for all 
} u \in \Wpzero{\mathcal{H}}.
		\end{align*}
	\end{enumerate}
\end{proposition}

\begin{proof}
	The proof of \textnormal{(i)} follows from Proposition \ref{proposition_embeddings} \textnormal{(iii)} and \textnormal{(vii)}. For \textnormal{(ii)}, let us assume the assertion is not true. Then there exists a sequence $\{u_n\}_{n \in \N} \subseteq \Wpzero{\mathcal{H}}$ such that
	\begin{equation*}
		\norm{u_n}_\mathcal{H} > n \norm{ \nabla u_n }_\mathcal{H}.
	\end{equation*}
	Let $ y_n := \frac{u_n}{\norm{u_n}_\mathcal{H}}$, then 
	\begin{equation*}
		1 > \frac{1}{n} > \norm{ \nabla y_n }_\mathcal{H} \quad \text{and} \quad
		\norm{ y_n }_\mathcal{H} = 1 \quad 
		\text{ for all }n \in \N,		
	\end{equation*} 
	i.e. the sequence $\{y_n\}_{n \in \N}$ is bounded in $\Wpzero{\mathcal{H}}$. Therefore, there exists a subsequence (not relabeled) and $y \in \Wpzero{\mathcal{H}}$ such that 
	\begin{equation*}
		y_n \weak y \quad 
		\text{in } \Wp{\mathcal{H}}.
	\end{equation*}
	By the weak lower semicontinuity of the mapping $v \mapsto \norm{ \nabla 
v }_\mathcal{H}$ on $\Wpzero{\mathcal{H}}$ (it is a convex, continuous mapping) there holds
	\begin{equation*}
		\norm{ \nabla y }_\mathcal{H} 
		\leq \liminf_{n \to \infty} \norm{ \nabla y_n }_\mathcal{H}
		\leq \lim_{n \to \infty} \frac{1}{n} = 0,
	\end{equation*}
	thus $y = c \in \R$ is a constant function. And by Proposition \ref{proposition_embeddings} \textnormal{(i)} we have that $y \in \Wpzero{p(\cdot)}$, where it is known that the only constant function is $y = 0$. However, this leads to a contradiction since by \textnormal{(i)}
	\begin{equation*}
		y_n \to y \quad 
		\text{in } \Lp{\mathcal{H}},
	\end{equation*}
	hence $\norm{y}_\mathcal{H} = \lim_{n \to \infty} \norm{y_n}_\mathcal{H} = 1$, so $y \neq 0$. 
\end{proof}

Based on Proposition \ref{prop-generalpoincare} we equip the space $\Wpzero{\mathcal{H}}$ with the norm
\begin{align*}
	\|u\|_{1,\mathcal{H},0}=\|\nabla u\|_{\mathcal{H}} \quad\text{for all }u \in \Wpzero{\mathcal{H}}.
\end{align*}

Similarly to Proposition \ref{proposition-W1H-unifconv}, we next prove that this norm is a uniformly convex norm on this space and satisfies the Radon–Riesz (or Kadec-Klee) property with respect to the modular.
\begin{proposition}
	Let hypotheses \textnormal{(H)} be satisfied.
	\begin{enumerate}
		\item[\textnormal{(i)}]	The norm $\|  \cdot \|_{ 1, \mathcal{H}, 0 }$ on $\Wpzero{\mathcal{H}}$ is uniformly convex. 
		\item[\textnormal{(ii)}]	For any sequence $\{u_n\}_{n \in \N} \subseteq \Wpzero{\mathcal{H}}$ such that 
		\begin{equation*}
			u_n \weak u \text{ in }\Wpzero{\mathcal{H}}\quad\text{and}\quad
			\rho_{\mathcal{H}} ( \nabla u_n ) \to \rho_{\mathcal{H}} ( \nabla u )
		\end{equation*}
		it holds that $u_n \to u$ in $\Wpzero{\mathcal{H}}$.
	\end{enumerate}
\end{proposition}

\begin{proof}
	The argument is analogous to the one of the proof of Proposition \ref{proposition-W1H-unifconv} with some extra considerations. First of all, consider the space $(\Lp{\mathcal{H}})^N$ equipped with the Luxemburg norm $\| \cdot \|_{\rho_{\mathcal{H},N}}$ given by the modular $\rho_{\mathcal{H},N} (u) = \rho_{\mathcal{H}} ( |u| )$ for all $u \in (\Lp{\mathcal{H}})^N$. By Theorem 2.4 of Fan-Guan \cite{Fan-Guan-2001}, the norm $\| \cdot \|_{\rho_{\mathcal{H},N}}$ on $(\Lp{\mathcal{H}})^N$ is uniformly convex, and there is an isometric embedding from $\Wpzero{\mathcal{H}}$ into $(\Lp{\mathcal{H}})^N$ given by $u \mapsto \nabla u$. This finishes the proof of \textnormal{(i)}. Regarding the assumptions (Q), take $Y = \Wpzero{\mathcal{H}} = X$ and the modular $\rho_{\mathcal{H}} ( \nabla \cdot )$, then (Q$_1$)-(Q$_6$) hold by Proposition \ref{proposition_modular_properties} and (Q$_7$) holds by an analogous argument to the proof of \textnormal{(viii)} in the same proposition. Fix now any $c > 0$. Then the $N$-function given by $c\mathcal{H}$ is uniformly convex and, by repeating the proof of \textnormal{(i)} for this $N$-function, one gets that the norm $\| \cdot \|_{1,c\mathcal{H},0}$ on $\Wpzero{\mathcal{H}}$ is uniformly convex. Note also that $\rho_{\mathcal{H}}( \nabla u ) \to 0$ if and only if $\rho_{c\mathcal{H}}( \nabla u ) \to 0$, so they generate the same topology. It is straightforward that (Q$_1$) and (Q$_2$) hold for $c\mathcal{H}$, and (Q$_3$)-(Q$_7$) follow doing an analogous argument to Proposition \ref{proposition_modular_properties}. Thus \textnormal{(ii)} follows from Theorem 3.5 of Fan-Guan \cite{Fan-Guan-2001}. 
\end{proof}

It is also possible to have a more general criterion for compact embeddings of $\Wp{\mathcal{H}}$ into some Musielak-Orlicz spaces. This will also 
have as a consequence a Poincar\'e inequality for $\Wpzero{\mathcal{H}}$. 
In order to do so we first need the definition of the Sobolev conjugate function of $\mathcal{H}$. We define for all $x\in \Omega$
\begin{align*}
	\mathcal{H}_1(x,t)=
	\begin{cases}
		t \mathcal{H}(x,1) &\text{if }0 \leq t \leq 1,\\
		\mathcal{H}(x,t)&\text{if } t > 1.
	\end{cases}
\end{align*}

Since $\Omega$ is a bounded, we know that $\Lp{\mathcal{H}}=\Lp{\mathcal{H}_1}$ and $\Wp{\mathcal{H}}=\Wp{\mathcal{H}_1}$, see Musielak \cite{Musielak-1983}. Therefore, for embedding results of $\Wp{\mathcal{H}}$ we 
may use $\mathcal{H}_1$ instead of $\mathcal{H}$. For simplification, we write $\mathcal{H}$ instead of $\mathcal{H}_1$.

\begin{definition}
	We denote by $\mathcal{H}^{-1}(x,\cdot)\colon [0,\infty) \to [0,\infty)$ 
for all $x \in \Omega$ the inverse function of $\mathcal{H}(x,\cdot)$. Furthermore, we define $\mathcal{H}_*^{-1}\colon \Omega \times [0,\infty) \to [0,\infty)$ by
	\begin{align*}
		\mathcal{H}^{-1}_*(x,s)=\int^s_0 \frac{\mathcal{H}^{-1}_1(x,\tau)}{\tau^{\frac{N+1}{N}}}\diff \tau \quad\text{for all }(x,s) \in \Omega\times [0,\infty),
	\end{align*}
	where $\mathcal{H}_*\colon (x,t) \in \Omega\times[0,\infty) \to s\in [0,\infty)$ is such that $\mathcal{H}^{-1}_*(x,s)=t$. The function $\mathcal{H}_*$ is called the Sobolev conjugate function of $\mathcal{H}$.
\end{definition}

We suppose the following stronger assumptions as that in \textnormal{(H)}.
\begin{enumerate}
	\item[\textnormal{(H')}]
		$\Omega\subseteq \R^N$, $N\geq 2$, is a bounded domain with Lipschitz boundary $\partial\Omega$, $0\leq \mu(\cdot)\in C^{0,1}(\close)$ and $p, q 
\in C^{0,1}(\close)$ are chosen such that
		\begin{enumerate}
			\item[\textnormal{(i)}] 
				$1 < p(x) < N$ and $p(x) < q(x)$ for all $x \in \close$;
			\item[\textnormal{(ii)}] 
				$\displaystyle \frac{q_+}{p_-}<1+\frac{1}{N}$,
		\end{enumerate}
\end{enumerate}

\begin{proposition}
	Let hypotheses \textnormal{(H')} be satisfied. Then the following hold:
	\begin{enumerate}
		\item[\textnormal{(i)}]
			$\Wp{\mathcal{H}}\hookrightarrow \Lp{\mathcal{H}_*}$ continuously;
		\item[\textnormal{(ii)}]
			Let $\mathcal{K}\colon\Omega \times [0,\infty)\to [0,\infty)$ be continuous such that $\mathcal{K} \in N(\Omega)$ and $\mathcal{K}\ll \mathcal{H}_*$, then $\Wp{\mathcal{H}}\hookrightarrow \Lp{\mathcal{K}}$ compactly;
		\item[\textnormal{(iii})]
			It holds $\mathcal{H} \ll \mathcal{H}_*$ and in particular, $\Wp{\mathcal{H}}\hookrightarrow \Lp{\mathcal{H}}$ compactly;
		\item[\textnormal{(iv)}]
			It holds
			\begin{align*}
				\|u\|_{\mathcal{H}} \leq C\|\nabla u\|_{\mathcal{H}}\quad\text{for all } u \in \Wpzero{\mathcal{H}},
			\end{align*}
			where $C>0$ is a constant independent of $u$.
	\end{enumerate}
\end{proposition}

\begin{proof}
	The proof of the proposition follows directly from Theorems 1.1 and 1.2 of Fan \cite{Fan-2012}, see also Proposition 2.18 of Colasuonno-Squassina 
\cite{Colasuonno-Squassina-2016}. We only need to prove (P4) and condition (2) of Proposition 3.1 in Fan \cite{Fan-2012}, that is, 
	\begin{equation} \label{condition_fan_p4}
		\lim_{t \to + \infty} \mathcal{H}^{-1}_*(x,t) = + \infty \quad \text{for all } x \in \Omega
	\end{equation}
	and there exist positive constants $\delta_1<\frac{1}{N}$, $c_1$ and $t_1$ such that
	\begin{align}\label{condition_fan_p5}
		\l| \frac{\partial \mathcal{H}(x,t)}{\partial x_j}\r| \leq c_1 (\mathcal{H}(x,t))^{1+\delta_1}\quad (j=1,\ldots, N)
	\end{align}
	for all $x\in \Omega$ and $t \geq t_1$ for which $\nabla \mu(x), \nabla p(x)$ and $\nabla q(x)$ exist and so $\frac{\partial \mathcal{H}(x,t)}{\partial x_j}$ does. First, note that for $(x,t) \in \mu^{-1}(\{0\}) \times 
[1,+\infty)$ it holds $\mathcal{H}_1^{-1}(x,t) = t^{\frac{1}{p(x)}}$ and hence
	\begin{align*}
		\mathcal{H}_*^{-1}(x,t) 
		& = \mathcal{H}_*^{-1}(x,1) + \int_1^t s^{\left( \frac{1}{p(x)} - \frac{N + 1}{N} \right) } \diff s \\
		& = \mathcal{H}_*^{-1}(x,1) + \frac{1}{ \frac{1}{p(x)} - \frac{N + 1}{N} + 1} \left[ t^{\left( \frac{1}{p(x)} - \frac{N + 1}{N} + 1 \right)} - 
1 \right] \xrightarrow{t \to +\infty} + \infty
	\end{align*}
	as $0 < \left[ \frac{1}{p(x)} - \frac{1}{N} \right] < 1$ due to $1 < p(x) < N$ for all $x \in \close$. For the rest of the points, i.e., $(x,t) \in \mu^{-1}((0,+\infty)) \times [1,+\infty)$, note that
	\begin{align*}
		\lim_{t \to + \infty} \frac{t^{p(x)} + \mu(x) t^{q(x)}}{t^{q(x)}} = \mu(x) \quad \text{for all } x \in \close.
	\end{align*}
	So for any $\eps > 0$ there exist some $K_x>1$ such that 
	\begin{align*}
		t^{p(x)} + \mu(x) t^{q(x)} < \left[ \eps + \mu(x) \right] t^{q(x)} \quad \text{for all } x \in \Omega\text{ and for all } t \geq K_x,
	\end{align*}
	and by inverting these strictly increasing functions
	\begin{align*}
		\mathcal{H}_1^{-1}(x,t) > \left( \frac{t}{\eps + \mu(x)} \right)^{\frac{1}{q(x)}} \quad \text{for all } x \in \Omega \text{ and for all } t \geq 
K_x,
	\end{align*}
	which yields the situation to repeat the argument of the integral above. 
Hence \eqref{condition_fan_p4} is satisfied.
	
	For the second condition, we can find $\eta>0$ small enough such that
	\begin{align}\label{Lipschitz_1}
		\frac{q_++\eta}{p_-}<1+\frac{1}{N},
	\end{align}
	see \textnormal{(H)(ii)}, and 
	\begin{align}\label{Lipschitz_2}
		\ln(t) \leq c t^\eta
	\end{align}
	with $c$ depending only on $\eta$ and $\ln$ being the natural logarithm. 
Denoting by $c_\mu, c_p, c_q$ the Lipschitz constants of $\mu, p, q$, respectively, we have for $t\geq 1$ by using \eqref{Lipschitz_2},
	\begin{align*}
		\l| \frac{\partial \mathcal{H}(x,t)}{\partial x_j}\r|
		& \leq t^{p(x)}\l|\frac{\partial p}{\partial x_j}(x)\r|\ln(t)+\l|\frac{\partial \mu}{\partial x_j}(x)\r|t^q(x)+\mu(x)t^{q(x)}\l|\frac{\partial q}{\partial x_j}(x)\r|\ln(t)\\
		&\leq \l(c_pc+c_\mu+c_qc \|\mu\|_\infty \r)t^{q(x)+\eta}\\
		& \leq \l(c_pc+c_\mu+c_qc \|\mu\|_\infty \r) \l(t^{p(x)}+\mu(x)t^{q(x)}\r)^{\frac{q_++\eta}{p_-}}.
	\end{align*}
	Then, condition \eqref{condition_fan_p5} is satisfied with
	\begin{align*}
		c_1= c_pc+c_\mu+c_qc \|\mu\|_\infty, \quad t_0\geq 1 \quad\text{and}\quad \delta_1=\frac{q_++\eta}{p_-}-1<\frac{1}{N},
	\end{align*}
	see also \eqref{Lipschitz_1}.
\end{proof}

Next, we want to answer the question when smooth functions are dense in $\Wp{\mathcal{H}}$. This result is of independent interest and our idea is 
to apply Theorem 6.4.7 of Harjulehto-H\"{a}st\"{o} \cite{Harjulehto-Hasto-2019}. First we recall some definitions stated in \cite{Harjulehto-Hasto-2019}.

\begin{definition}$~$
	\begin{enumerate}
		\item[\textnormal{(i)}]
			We call a function $g\colon (0,\infty)\to\R$ almost increasing if there exists a constant $a \geq 1$ such that $g(s) \leq ag(t)$ for all $0<s<t$. Similarly, we define almost decreasing functions.
		\item[\textnormal{(ii)}]
			We say that $\ph\colon \Omega\times [0,\infty)\to[0,\infty]$ is a $\Phi$-prefunction if $x \mapsto \ph(x, |f(x)|)$ is measurable for every measurable function $f\colon \Omega\to \R$, $\ph(x,0)=0$,
			\begin{align*}
				\lim_{t\to 0^+}\ph(x,t)=0 \quad\text{and}\quad \lim_{t\to\infty}\ph(x,t)=\infty \quad\text{for a.\,a.\,}x\in\Omega.
			\end{align*}
			If in addition the condition
			\begin{align*}
				\frac{\ph(x,t)}{t} \text{ is almost increasing for a.\,a.\,}x\in\Omega
			\end{align*}
			is satisfied, then $\ph$ is called a weak $\Phi$-function and the class of all weak $\Phi$-functions is denoted by $\Phi_{\text{w}}(\Omega)$.
		\item[\textnormal{(iii)}]
			We say that $\ph \in \Phi_{\textnormal{w}}(\Omega)$ satisfies \textnormal{(A0)}, if there exists a constant $\beta\in (0,1]$ such that $\beta \leq \ph^{-1}(x,1) \leq \frac{1}{\beta}$ for a.\,a.\,$x \in \Omega$.
		\item[\textnormal{(iv)}]
			We say that $\ph \in \Phi_{\textnormal{w}}(\Omega)$ satisfies \textnormal{(A1)}, if there exists $\beta \in (0,1)$ such that
			\begin{align*}
				\beta \ph^{-1}(x,t) \leq \ph^{-1}(y,t)
			\end{align*}
			for every $t \in [1,\frac{1}{|B|}]$, for a.\,a.\,$x,y\in B\cap \Omega$ 
and for every ball $B$ with $|B|\leq 1$.
		\item[\textnormal{(v)}]
			We say that $\ph \in \Phi_{\textnormal{w}}(\Omega)$ satisfies \textnormal{(A1')}, if there exists $\beta \in (0,1)$ such that
			\begin{align*}
				\ph(x,\beta t) \leq \ph(y,t)
			\end{align*}
			for every $\ph (y,t) \in [1,\frac{1}{|B|}]$, for a.\,a.\,$x,y\in B\cap 
\Omega$ and for every ball $B$ with $|B|\leq 1$.
		\item[\textnormal{(vi)}]
			We say that $\ph \in \Phi_{\textnormal{w}}(\Omega)$ satisfies \textnormal{(A2)}, if for every $s>0$ there exist $\beta\in (0,1]$ and $h\in \Lp{1}\cap \Linf$ such that
		 	\begin{align*}
		 		\beta \ph^{-1}(x,t) \leq \ph^{-1}(y,t)
		 	\end{align*}
	 		for a.\,a.\,$x\in\Omega$ and for all $t \in [h(x)+h(y),s]$.
	 	\item[\textnormal{(vii)}]
	 		We say that $\ph\colon \Omega\times (0,\infty)\to\R$ satisfies \textnormal{(aDec)} if there exists $\ell \in (0,
	 		\infty)$ such that
	 		\begin{align*}
	 			\frac{\ph(x,t)}{t^\ell} \text{ is almost decreasing for a.\,a.\,}x\in\Omega.
	 		\end{align*}
	\end{enumerate}
\end{definition}

In the sequel we will use $f\approx g$ and $f \lesssim g$  if there exist
constants $c_1, c_2>0$ such that $c_1f\leq g\leq c_2f$ and $f\leq c_2g$, respectively.

Now we are ready to prove the density of smooth functions in the Musielak-Orlicz Sobolev space $\Wp{\mathcal{H}}$.

\begin{theorem}\label{theorem_density_1}
	Let hypotheses \textnormal{(H')} be satisfied, where  \textnormal{(H')}\textnormal{(ii)} is replaced by
	\begin{align}\label{density_weaker}
		\frac{q_+}{p_-}\leq 1+\frac{1}{N}.
	\end{align}
	 Then $C^\infty(\Omega)\cap\Wp{\mathcal{H}}$ is dense in $\Wp{\mathcal{H}}$.
\end{theorem}

\begin{proof}
	We are going to apply Theorem 6.4.7 of Harjulehto-H\"{a}st\"{o} \cite{Harjulehto-Hasto-2019}. First note that $\mathcal{H}\in \Phi_{\textnormal{w}}(\Omega)$ (see also Definitions 2.1.2, 2.5.1 and 2.5.2 in \cite{Harjulehto-Hasto-2019}). Furthermore we have
	\begin{align*}
		1 \leq \mathcal{H}(x,2) \leq \l(2^{p_+} +\|\mu\|_\infty 2^{q_+}\right) \cdot 1 \quad\text{for all }x\in \Omega.
	\end{align*}
	Hence $\mathcal{H}(x,2) \approx 1$ and so we can apply Corollary 3.7.5 in \cite{Harjulehto-Hasto-2019} which shows that $\mathcal{H}$ satisfies condition \normalfont{(A0)}. For $t \in [0,s]$ we have
	\begin{align*}
		t^{p(x)} + \mu(x) t^{q(x)} \approx t^{p(x)},
	\end{align*}
	where the constants depend on $s$. Hence, Lemma 4.2.5 in \cite{Harjulehto-Hasto-2019} implies that condition \normalfont{(A2)} is satisfied. Moreover, by Lemma 2.2.6 in \cite{Harjulehto-Hasto-2019}, we know that \normalfont{(aDec)} is satisfied since $\mathcal{H}\in \Phi_{\textnormal{w}}(\Omega)$  fulfills the $\Delta_2$-condition, see \eqref{Delta2-condition}. 
	
	It remains to show that \normalfont{(A1)} is satisfied. First we note that since $\mathcal{H}\in \Phi_{\textnormal{w}}(\Omega)$ satisfies \normalfont{(A0)}, we know that it is enough to show that $\mathcal{H}\in \Phi_{\textnormal{w}}(\Omega)$ fulfills \normalfont{(A1')}, see \cite[Corollary 
4.1.6]{Harjulehto-Hasto-2019}. Therefore, we need to show that there exists $\beta \in (0,1)$ such that
	\begin{align*}
		\mathcal{H}(x,\beta t) \leq \mathcal{H}(y,t)
	\end{align*}
	for every $\mathcal{H} (y,t) \in [1,\frac{1}{|B|}]$, for a.\,a.\,$x,y\in 
B\cap \Omega$ and for every ball $B$ with $|B|\leq 1$.
	
	For this purpose let us fix a ball $B\subseteq \mathbb{R}^N$ of radius $r>0$, such that $|B|\le 1$ (in particular $r<1$).
	We know that $|B|=\alpha(N)r^N$, where $\alpha(N)>1$ is a constant depending only on the dimension $N$.
	Note that the condition 
	\begin{align}\label{cond_1}
		\mathcal{H}(y,t)=t^{p(y)}+\mu(y) t^{q(y)}\in \l[1,\frac{1}{\alpha(N)r^N}\r]
	\end{align}
	implies that
	\begin{align*}
		t^{p_-}\l(1+\|\mu\|_{\infty}\r)= t^{p_-}+\|\mu\|_{\infty}t^{p_-}\ge
		\begin{cases}
			t^{p(y)}+\mu(y) t^{q(y)} \geq 1, & \text{if } t\leq 1,\\
			1, & \text{if } t\geq 1,
		\end{cases}
	\end{align*}
	and
	\begin{align*}
		t^{p_-}\leq
		\begin{cases}
			t^{p(y)}\le t^{p(y)}+\mu(y) t^{q(y)}\le \frac{1}{\alpha(N)r^N}, & \text{if } t\ge 1,\\
			\frac{1}{\alpha(N)r^N},& \text{if } t\leq 1.
		\end{cases}
	\end{align*}
	Note that the last inequality for $t\leq 1$ is always true since $t^{p_-}\leq 1\leq \frac{1}{|B|}$ as $|B|\leq 1$. Hence, \eqref{cond_1} implies in particular that
	\begin{equation}\label{eq_new_01}
		t\in\l[ (1+\|\mu\|_{\infty})^{-\frac{1}{p_-}}, \alpha(N)^{-\frac{1}{p_-}} r^{-\frac{N}{p_-}}\r].
	\end{equation}
	
	Now, it is enough to show that there exists $\beta\in (0,1)$ such that
	\begin{equation}\label{eq_new_02}
		(\beta t)^{p(x)}+\mu(x)(\beta t)^{q(x)}\le t^{p(y)}+\mu(y) t^{q(y)}
	\end{equation}
	for all $t$ satisfying \eqref{eq_new_01} and almost all $x,y\in \Omega$ such that $|x-y|\le 2r$.
	
	{\bf Claim:}
	For all $t$ satisfying \eqref{eq_new_01} and almost all $x,y\in \Omega$ such that $|x-y|\le 2r$
	we have
	\begin{equation}\label{eq_new_03}
		t^{p(x)}\le M\cdot t^{p(y)}\quad\textrm{and}\quad t^{q(x)}\le M\cdot t^{q(y)},
	\end{equation}
	for some constant $M=M(N,p,q,\mu)>0$ not depending on $x,y,t$. 
	
	So, let us fix $t$ satisfying \eqref{eq_new_01} and $x,y\in \Omega$ such 
that $|x-y|\le 2r$. Since $p \in C^{0,1}(\close)$ we have
	\begin{equation}\label{eq_new_04}
		|p(x)-p(y)|\le c_p|x-y|\le 2rc_p,
	\end{equation}  
	where $c_p>0$ denotes the Lipschitz constant of the function $p$. 

	{Case I:}
	If $t\le 1$ and $p(x)\ge p(y)$ or $t\ge 1$ and $p(x)\le p(y)$ then the first inequality in \eqref{eq_new_03} holds with $M=1$.

	{Case II:} 
	If $t\le 1$ and $p(x)\le p(y)$ then by applying \eqref{eq_new_01} it follows that
	\begin{align*}
		t^{p(x)}=t^{p(x)-p(y)}t^{p(y)}
		\leq \l((1+\|\mu\|_{\infty})^{\frac{1}{p_-}}\r)^{p(y)-p(x)}t^{p(y)} \leq 
		\l(1+\|\mu\|_{\infty}\r)^{\frac{p_+}{p_-}}t^{p(y)}.
	\end{align*}
	Thus the first inequality in \eqref{eq_new_03} hold with $M=(1+\|\mu\|_{\infty})^{\frac{p_+}{p_-}}$.
	
	{Case III:}
	If $t\ge 1$ and $p(x)\ge p(y)$ we have by using \eqref{eq_new_01} and \eqref{eq_new_04}
	\begin{align*}
		t^{p(x)} = t^{p(x)-p(y)}t^{p(y)}
		\leq
		\l(\alpha(N)^{-\frac{1}{p_-}} r^{-\frac{N}{p_-}}\r)^{2rc_p}t^{p(y)}
		\leq
		\l(\alpha(N)^{-\frac{2c_p}{p_-}}\r)^r(r^r)^{-\frac{2Nc_p}{p_-}}t^{p(y)}.
	\end{align*}
	Note that the function $\delta(r)=\big(\alpha(N)^{-\frac{2c_p}{p_-}}\big)^r(r^r)^{-\frac{2Nc_p}{p_-}}$ is strictly positive and continuous on the interval $\l[0,\frac{1}{\alpha(N)^{\frac{1}{N}}}\r]$ where $\delta(0)=1$. Hence it attains its maximum at some $r_0\in \l[0,\frac{1}{\alpha(N)^{\frac{1}{N}}}\r]$. 
	Then the first inequality in \eqref{eq_new_03} holds for $M=\delta(r_0)>0$. The second inequality in \eqref{eq_new_03} can be done in an analogous way again via three cases. Taking $M$ as the maximum of the six cases 
shows the assertion of the Claim. 

	Let us now prove  \eqref{eq_new_02}. Since $\mu \in C^{0,1}(\close)$ and 
$|x-y|\le 2r$ we have
	\begin{align}\label{eq_new_05}
		|\mu(x)-\mu(y)|\le c_{\mu}|x-y|\le 2c_{\mu}r,
	\end{align}
	where $c_{\mu}>0$ is the Lipschitz constant of the function $\mu$.

	Let us start with the left hand side of \eqref{eq_new_02}. Since $\beta \in (0,1)$ and taking \eqref{eq_new_03} as well as \eqref{eq_new_05} into 
account, we get
	\begin{align}\label{eq_new_06}
		\begin{split}
			(\beta t)^{p(x)}+\mu(x)(\beta t)^{q(x)}
			&\leq \beta^{p_-}t^{p(x)}+\mu(x)\beta^{p_-}t^{q(x)}\\
			& \leq \beta^{p_-}M\l(t^{p(y)}+\mu(x)t^{q(y)}\r)\\
			& \leq \beta^{p_-}M\l(t^{p(y)}+\mu(y)t^{q(y)}+2c_{\mu}r t^{q(y)}\r)\\
			& \le \beta^{p_-}M\l(t^{p(y)}+2c_{\mu}r t^{q(y)}\r)+\mu(y)t^{q(y)},
		\end{split}
	\end{align}
	where the last inequality holds providing $\beta<M^{-\frac{1}{p_-}}$.
	Continuing \eqref{eq_new_06} and applying \eqref{eq_new_01} we have
	\begin{align}\label{eq_new_07}
		\begin{split}
			(\beta t)^{p(x)}+\mu(x)(\beta t)^{q(x)}
			& \le \beta^{p_-}Mt^{p(y)}\l(1+2c_{\mu}r t^{q(y)-p(y)}\r)+\mu(y)t^{q(y)}\\
			& \leq \beta^{p_-}Mt^{p(y)}\l(1+2c_{\mu}r \l(\alpha(N)^{-\frac{1}{p_-}} r^{-\frac{N}{p_-}}\r)^{q_+-p_-}\r)+\mu(y)t^{q(y)}\\
			& = \beta^{p_-}Mt^{p(y)}\l(1+2c_{\mu} \alpha(N)^{-\frac{q_+}{p_-}+1} 
r^{1+N-N\frac{q_+}{p_-}}\r)+\mu(y)t^{q(y)}.
		\end{split}	
	\end{align}
	From \eqref{density_weaker} we have $1+N-N\frac{q_+}{p_-}\geq 0$. Using this we may continue \eqref{eq_new_07} since $r<1$ as follows
	\begin{align*}
		\begin{split}
			(\beta t)^{p(x)}+\mu(x)(\beta t)^{q(x)}\
			\leq
			\beta^{p_-}Mt^{p(y)}\l(1+2c_{\mu} \alpha(N)^{-\frac{q_+}{p_-}+1} \r)+\mu(y)t^{q(y)}.
		\end{split}	
	\end{align*}
	Choosing $\beta>0$ small enough, namely 
	\begin{align*}
		\beta< M^{-\frac{1}{p_-}}\l(1+2c_{\mu} \alpha(N)^{-\frac{q_+}{p_-}+1}\r)^{-\frac{1}{p_-}},
	\end{align*} 
	we see that \eqref{eq_new_02} holds. Note that the choice of $\beta$ depends only on $N,p,q$ and $\mu$. Therefore, $\mathcal{H}\in \Phi_{\textnormal{w}}(\Omega)$ satisfies \normalfont{(A1')} and so \normalfont{(A1)}. The assertion of the proposition follows now from Theorem 6.4.7 of Harjulehto-H\"{a}st\"{o} \cite{Harjulehto-Hasto-2019}.
\end{proof}

A careful reading of the proof of Theorem \ref{theorem_density_1} shows that the boundedness of $\Omega$ is not used. This leads to the following result.

\begin{theorem}\label{theorem_density_1_unbounded}
	Let hypotheses \textnormal{(H')} be satisfied, where $\Omega\subseteq \R^N$, $N\geq 2$ is an unbounded domain, $0\leq \mu(\cdot) \in \Linf \cap C^{0,1}(\close)$ and condition \textnormal{(H)(ii)} is replaced by
	\begin{align*}
		\frac{q_+}{p_-}\leq 1+\frac{1}{N}.
	\end{align*}
	Then $C^\infty(\Omega)\cap\Wp{\mathcal{H}}$ is dense in $\Wp{\mathcal{H}}$.
\end{theorem}

Next, we are going to prove the density under weaker assumptions as in the Theorems \ref{theorem_density_1} and \ref{theorem_density_1_unbounded}. 
First we recall the following definition.

\begin{definition}
	We say that a function $g\colon\Omega \to \R$ satisfies the log-H\"older 
decay condition if there exists $g_\infty\in\R$ and a constant $c_g>0$ such that
	\begin{align*}
		|g(x)-g_\infty| \leq \frac{c_g}{\log(e+|x|)}\quad\text{for all }x\in\Omega.
	\end{align*}
\end{definition}

We suppose the following conditions which are weaker than \textnormal{(H')}:
\begin{enumerate}
	\item[\textnormal{(H'')}]
	$\Omega\subseteq \R^N$, $N\geq 2$, is an unbounded domain, $0\leq \mu(\cdot)\in \Linf$ and $p\colon \Omega \to [1, \infty)$, $q\colon\Omega \to [1, \infty)$ are bounded functions that are log-H\"older continuous and satisfy the log-H\"older decay condition with $p(x)\leq q(x)$ for all $x\in\Omega$.
\end{enumerate}

We start with a characterization of condition \textnormal{(A1)}.

\begin{proposition}\label{prop_A1}
	Let hypotheses \textnormal{(H'')} be satisfied. Then $\mathcal{H}(x, t) = t^{p(x)} + \mu(x) t^{q(x)}$ satisfies condition \textnormal{(A1)} if and only if there exists a constant $\beta >0$ such that
	\begin{align*}
		\beta \mu(y)^{\frac1{q(y)}}\le |x-y|^{N\l(\frac1{p(y)}-\frac1{q(y)}\r)}+ \mu(x)^{\frac{1}{q(x)}}
	\end{align*}
	for every $x, y \in \Omega$.
\end{proposition}

\begin{proof}
	Let us first observe that $\mathcal{H}(x, t) \approx \max\{t^{p(x)}, \mu(x) t^{q(x)} \}$ and hence
	\begin{align*}
		\mathcal{H}^{-1}(x, t) \approx \min\l\{ t^{\frac{1}{p(x)}}, \l(\frac{t}{\mu(x)}\r)^{\frac{1}{q(x)}}\r\}.
	\end{align*}

	After dividing by $t^{\frac{1}{q(y)}}$ condition \textnormal{(A1)} becomes
	\begin{align*}
		\beta \min\l\{t^{\frac1{p(x)} - \frac1{q(y)}}, \mu(x)^{-\frac1{q(x)}} t^{\frac1{q(x)}- \frac1{q(y)}}\r\} \leq \min\l\{ t^{\frac1{p(y)} - \frac1{q(y)}}, \mu(y)^{-\frac1{q(y)}}\r\}
	\end{align*}
	for  $x,y \in B\cap \Omega$, $|B| \leq 1$ and $t \in \l[1, \frac{1}{|B|}\r]$. 

	If $p(x) \le p(y)$, then
	\begin{align*}
		t^{\frac1{p(x)}} \le t^{\frac1{p(x)} - \frac1{p(y)}} t^{\frac1{p(y)}} \le |B|^{\frac1{p(y)}-\frac1{p(x)}} t^{\frac1{p(y)}} \lesssim   t^{\frac1{p(y)}},
	\end{align*}
	where we have used the  $\log$-H\"older continuity of $\frac{1}{p}$ and $t \in \l[1, \frac{1}{|B|}\r]$, see Lemma~4.16 in Diening-Harjulehto-H\"{a}st\"{o}-R$\mathring{\text{u}}$\v{z}i\v{c}ka \cite{Diening-Harjulehto-Hasto-Ruzicka-2011}. If $p(x) \ge p(y)$, then $t^{\frac1{p(x)}} \le t^{\frac1{p(y)}}$ since $t \ge 1$.
	Similarly we obtain  by the $\log$-H\"older continuity that $t^{\frac1{q(x)}- \frac1{q(y)}} \le C$. 

	Thus for \textnormal{(A1)} we need to verify that
	\begin{align*}
		\beta \min\l\{t^{\frac1{p(y)} - \frac1{q(y)}}, \mu(x)^{-\frac1{q(x)}} \r\} \le \min\l\{ t^{\frac1{p(y)} - \frac1{q(y)}}, \mu(y)^{-\frac1{q(y)}}\r\}
	\end{align*}
	for  $x,y \in B\cap \Omega$, $|B| \le 1$ and $t \in \l[1, \frac{1}{|B|}\r]$. 
	
	We may assume that $\diam(B) \le 2 |x-y|$.
	The case $\mu(y)^{-\frac{1}{q(y)}} \ge t^{\frac{1}{p(y)}-\frac{1}{q(y)}}$ is trivial, so the condition 
	is equivalent to 
	\begin{align*}
		\beta \min\l\{ t^{\frac1{p(y)}-\frac1{q(y)}}, \mu(x)^{-\frac1{q(x)}}\r \} \le \mu(y)^{-\frac1{q(y)}}
	\end{align*}
	for $\mu(y)^{-\frac{1}{q(y)}} < t^{\frac{1}{p(y)}-\frac{1}{q(y)}}$. Since the exponent of $t$ is positive, we only need to check the inequality for the upper bound of $t$, namely $t=\frac1{|B|}$. 
	Moreover, $|B| \approx |x-y|^N$. Thus the condition is further equivalent to 
	\begin{align*}
		\beta \min\l\{ |x-y|^{-N\l(\frac1{p(y)}-\frac1{q(y)}\r)}, \mu(x)^{-\frac1{q(x)}} \r\} \le \mu(y)^{-\frac1{q(y)}},
	\end{align*}
	that is, equivalent to
	\begin{align*}
		\beta \mu(y)^{\frac{1}{q(y)}}\leq \max\l\{ |x-y|^{N\l(\frac1{p(y)}-\frac1{q(y)}\r)}, \mu(x)^{\frac{1}{q(x)}} \r\}
		\approx |x-y|^{N\l(\frac1{p(y)}-\frac1{q(y)}\r)}+ \mu(x)^{\frac{1}{q(x)}}.
	\end{align*}
\end{proof}

Next we can give a sufficient condition for $\mathcal{H}(\cdot,\cdot)$ to 
satisfy assumption \normalfont{(A1)}.

\begin{proposition}\label{condition_A1}
	Let hypotheses \textnormal{(H'')} be satisfied and let in addition $q\colon \Omega\to [1, \infty)$ be $\frac{\alpha}{q_{-}}$-H\"older continuous and $\mu\colon\Omega \to [0, \infty)$ be $\alpha$-H\"older continuous. 
	If $\frac{q(x)}{p(x)} \le 1 + \frac{\alpha}{N}$, then $\mathcal{H}(x, t) 
= t^{p(x)} + \mu(x) t^{q(x)}$ satisfies condition \textnormal{(A1)}.
\end{proposition}

\begin{proof}
	From Proposition \ref{prop_A1} we know that \textnormal{(A1)} holds with 
$\beta =1$ if
	\begin{align*}
		\l|\mu(x)^{\frac{1}{q(x)}} - \mu(y)^{\frac{1}{q(y)}}\r| \leq |x-y|^{N \gamma} 
	\end{align*}
	for all $x$ and $y$ with $|x-y|\le 1$ where $\gamma =\max\big\{\frac1{p(y)} - \frac1{q(y)}, \frac1{p(x)} - \frac1{q(x)} \Big\}$.
	We may assume that $\frac1{p(y)} - \frac1{q(y)} \ge \frac1{p(x)} - \frac1{q(x)}$.
	First we  use the triangle inequality to obtain
	\begin{align}\label{hoelder_one}
		\begin{split}
			\l|\mu(x)^{\frac{1}{q(x)}} - \mu(y)^{\frac{1}{q(y)}}\r| 
			&= \l|\mu(x)^{\frac{1}{q(x)}} - \mu(x)^{\frac{1}{q(y)}} + \mu(x)^{\frac{1}{q(y)}}- \mu(y)^{\frac{1}{q(y)}}\r|\\
			&\leq \l|\mu(x)^{\frac{1}{q(x)}} - \mu(x)^{\frac{1}{q(y)}}\r| 
			+ \l|\mu(x)^{\frac{1}{q(y)}}- \mu(y)^{\frac{1}{q(y)}}\r|.
		\end{split}
	\end{align}
	We estimate the first term on the right-hand side of \eqref{hoelder_one}. To this end, let $f(t) = a^t$.  Then, by the mean value theorem, 
	$f(v) - f(u) = f'(\xi) (v-u)$ for some $\xi$ between $u$ and $v$.  We choose $a=\mu(x)$, $u =\frac{1}{q(x)}$ and 
	$v = \frac{1}{q(y)}$. Then $a \in [0, \|\mu\|_\infty]$ and $u, v \in [\frac{1}{q_+}, 1]$. 
	 
	 Next we show that $|f'(\xi)|$ is bounded. If $a\ge 1$, then
	 \begin{align*}
	 	\l|f'(\xi)\r| = a^\xi \ln (a) \le \|\mu\|_\infty \ln(\|\mu\|_\infty).
	 \end{align*}
	 For $a \in [0,1)$ we obtain 
	 \begin{align*}
	 	\l|f'(\xi)\r| = - a^\xi \ln (a) \le - a^{\frac{1}{q_+}} \ln (a).
	 \end{align*} 

	A simple calculation shows that $a \mapsto - a^{\frac{1}{q_+}} \ln (a)$ got it largest value in $[0, 1]$ at $e^{-q_+}$. Hence $|f'(\xi)| \leq  \frac{q_+}e$. Thus we have
	\begin{align}\label{hoelder_two}
		\begin{split}
			\l|\mu(x)^{\frac{1}{q(x)}} - \mu(x)^{\frac{1}{q(y)}}\r| 
			&= \l|f(v) - f(u)\r|
			\leq c\l| \frac{1}{q(x)} - \frac{1}{q(y)} \r|\\
			&\leq c|q(y) - q(x)| \leq c  c_q |x-y|^{\frac{\alpha}{q_-}},
		\end{split}
	\end{align}
	where $c$ is a constant depending on $a$ and $q$, the constant $c_q$ is from the $\frac{\alpha}{q_{-}}$-H\"older continuity of $q$ and $|x-y|\leq 
1$.
	
	From $\frac{q(x)}{p(x)} \le 1 + \frac{\alpha}N$ we obtain
	\begin{align}\label{hoelder_3}
		N \l(\frac1{p(y)} - \frac1{q(y)}\r) 
		\leq \frac{N}{q(y)}\l(\frac{q(y)}{p(y)} - 1\r)
		\leq \frac{N}{q(y)} \frac{\alpha}{N} 
		\leq \frac{\alpha}{q_-}.
	\end{align}
	Combining \eqref{hoelder_two} and \eqref{hoelder_3} gives
	\begin{align*}
		\begin{split}
			\l|\mu(x)^{\frac{1}{q(x)}} - \mu(x)^{\frac{1}{q(y)}}\r| 
			& \leq c\l(\|a\|_\infty, q_+\r)  c_q |x-y|^{N \l(\frac1{p(y)} - \frac1{q(y)}\r)}.
		\end{split}
	\end{align*}
	
	Let us now estimate the second term of the right-hand side of \eqref{hoelder_one}.  We use the inequality $|x^r- y^r| \le |x-y|^r$, where $x, y \ge 0$ and $r \in(0,1]$, in order to obtain
	\begin{align}\label{hoelder_4}
		\begin{split}
			\l|\mu(x)^{\frac{1}{q(y)}}- \mu(y)^{\frac{1}{q(y)}}\r|
			&\leq |\mu(x) - \mu(y)|^{\frac{1}{q(y)}} 
			\leq c_\mu^{\frac1{q_-}} |x-y|^{\frac{\alpha}{q(y)}},
		\end{split}
	\end{align}
	where the constant $c_\mu$ is from the $\alpha$-H\"older continuity of $\mu$.
	
	From $\frac{q(x)}{p(x)} \le 1 + \frac{\alpha}N$ we obtain
	\begin{align}\label{hoelder_5}
		N \l(\frac1{p(y)} - \frac1{q(y)}\r) 
		\leq \frac{N}{q(y)}\l(\frac{q(y)}{p(y)} - 1\r)
		\leq \frac{N}{q(y)} \frac{\alpha}{N} 
		= \frac{\alpha}{q(y)}.
	\end{align}
	Combining \eqref{hoelder_4} and \eqref{hoelder_5} gives
	\begin{align*}
		\begin{split}
			\l|\mu(x)^{\frac{1}{q(y)}}- \mu(y)^{\frac{1}{q(y)}}\r|
			& \leq c_\mu^{\frac1{q_-}} |x-y|^{N \l(\frac1{p(y)} - \frac1{q(y)}\r)}.
		\end{split}
	\end{align*}
\end{proof}

Now we are ready to prove the density of the smooth functions in the Musielak-Orlicz Sobolev space $\Wp{\mathcal{H}}$ when $\Omega$ is unbounded under the assumptions \textnormal{(H'')}.

\begin{theorem}\label{theorem_density_2}
	Let hypotheses \textnormal{(H'')} be satisfied and let in addition $q\colon \Omega\to [1, \infty)$ be $\frac{\alpha}{q_{-}}$-H\"older continuous and $\mu\colon\Omega \to [0, \infty)$ be $\alpha$-H\"older continuous. 
	If $\frac{q(x)}{p(x)} \le 1 + \frac{\alpha}{N}$, then $C^\infty(\Omega)\cap\Wp{\mathcal{H}}$ is dense in $\Wp{\mathcal{H}}$.
\end{theorem}

\begin{proof}
	As in Theorem \ref{theorem_density_1} we will show the result by applying Theorem 6.4.7 of Harjulehto-H\"{a}st\"{o} \cite{Harjulehto-Hasto-2019}. 
In the same way as in Theorem \ref{theorem_density_1} we know that $\mathcal{H}\in \Phi_{\textnormal{w}}(\Omega)$ fulfills \normalfont{(A0)} , \normalfont{(A2)} and  \normalfont{(aDec)}. Finally, from Proposition \ref{condition_A1} we know that \normalfont{(A1)} is satisfied and the assertion of the theorem follows.
\end{proof}

\begin{remark}
	Note that for bounded domains the log-H\"older condition (local condition) in \textnormal{(H'')} is enough, we do not need the log-H\"older decay 
condition (condition near infinity).
\end{remark}

Comparison the assumptions of Theorems \ref{theorem_density_1} and \ref{theorem_density_2} for bounded domains with Lipschitz boundary, we see that the assumptions in Theorem \ref{theorem_density_2} are weaker than those in Theorem \ref{theorem_density_1}. Indeed the Lipschitz continuity can 
be replaced by certain H\"older or log-H\"older conditions and the inequality \eqref{density_weaker} implies
\begin{align*}
	\frac{q(x)}{p(x)} \le 1 + \frac{1}{N}\quad\text{for all }x\in\close \text{ and for }\alpha=1.
\end{align*}
In the unbounded case the situation is a bit different. The assumptions of Theorem \ref{theorem_density_1_unbounded} imply the ones of Theorem \ref{theorem_density_2} except the log-H\"older decay condition. Indeed, Lipschitz continuity does not imply the log-H\"older decay condition.\\

Let us now comment on the well-known eigenvalue problem for the $r$-Laplacian with homogeneous Dirichlet boundary condition and $1<r<\infty$ defined by
\begin{equation}\label{eigenvalue_problem}
\begin{aligned}
	-\Delta_r u& =\lambda|u|^{r-2}u\quad && \text{in } \Omega,\\
	u & = 0  &&\text{on } \partial \Omega.
\end{aligned}
\end{equation}

It is known that the first eigenvalue $\lambda_{1,r}$ of \eqref{eigenvalue_problem} is positive, simple, and isolated. Moreover, it can be variationally characterized through
\begin{align}\label{lambda_one}
	\lambda_{1,r} = \inf_{u \in W^{1,r}(\Omega)} \left \{\int_\Omega |\nabla u|^r \diff x : \int_\Omega |u|^r \diff x=1 \right \},
\end{align}
see L{\^e} \cite{Le-2006}. We will make use of the first eigenvalue in the statements of Theorems \ref{theorem_existence} and \ref{theorem_uniqueness}.

We now recall some definitions that we will use in the sequel.

\begin{definition}\label{SplusPM}
	Let $X$ be a reflexive Banach space, $X^*$ its dual space and denote by $\langle \cdot \,, \cdot\rangle$ its duality pairing. Let $A\colon X\to X^*$, then $A$ is called
	\begin{enumerate}[leftmargin=1cm]
		\item[\textnormal{(i)}]
			to satisfy the $\Ss$-property if $u_n \weak u$ in $X$ and $\limsup_{n\to \infty} \langle Au_n,u_n-u\rangle \leq 0$ imply $u_n\to u$ in $X$;
		\item[\textnormal{(ii)}]
			pseudomonotone if $u_n \weak u$ in $X$ and $\limsup_{n\to \infty} \langle Au_n,u_n-u\rangle \leq 0$ imply 
			\begin{align*}
				\liminf_{n \to \infty} \langle Au_n,u_n-v\rangle \geq \langle Au , u-v\rangle \quad\text{for all }v \in X;
			\end{align*}
		\item[\textnormal{(iii)}]
			coercive if there exists some function $g\colon[0,\infty) \to \R$ such 
that $\lim_{t \to + \infty} g(t) = + \infty$ and
			\begin{align*}
				\frac{\langle Au, u \rangle}{\|u\|_X} \geq g(\| u \|_X ) \text{ for all } u \in X.
			\end{align*}
	\end{enumerate}
\end{definition}

\begin{remark} \label{RemarkPseudomonotone}
	Note that if the operator $A\colon X \to X^*$ is bounded, then the definition of pseudomonotonicity in Definition \ref{SplusPM} \textnormal{(ii)} 
is equivalent to $u_n \weak u$ in $X$ and $\limsup_{n\to \infty} \langle Au_n,u_n-u\rangle \leq 0$ imply $Au_n \weak Au$ and $\langle Au_n,u_n\rangle \to \langle Au,u\rangle$. We will use this equivalent condition for bounded operators in Section \ref{section_4}.
\end{remark}

Our existence result is based on the following surjectivity result for pseudomonotone operators, see, for example, Papageorgiou-Winkert \cite[Theorem 6.1.57]{Papageorgiou-Winkert-2018}.

\begin{theorem}\label{theorem_pseudomonotone}
	Let $X$ be a real, reflexive Banach space, let $A\colon X\to X^*$ be a pseudomonotone, bounded, and coercive operator, and $b\in X^*$. Then, a solution of the equation $Au=b$ exists.
\end{theorem}

\section{Properties of the variable exponent double phase operator}\label{section_3}

In this section we introduce the new double phase operator $A$ related to 
our problem \eqref{problem8} and its corresponding energy functional $I$ given in \eqref{integral_minimizer}. To this end, let $A\colon \Wpzero{\mathcal{H}}\to \Wpzero{\mathcal{H}}^*$ be given by
\begin{align*}
	\l\lan A(u),v \r\ran_{\mathcal{H}}=\into \l( |\nabla u|^{p(x)-2}\nabla 
u \cdot \nabla v +\mu(x) |\nabla u|^{q(x)-2}\nabla u \cdot \nabla v\r)\diff x,
\end{align*}
for all $u,v \in \Wpzero{\mathcal{H}}$, where $\lan\cdot\,,\cdot\ran_{\mathcal{H}}$ denotes the duality pairing between $\Wpzero{\mathcal{H}}$ and 
its dual space $\Wpzero{\mathcal{H}}^*$. As mentioned in the Introduction, the energy functional $I\colon \Wpzero{\mathcal{H}}\to\R$ related to $A$ is given  by
\begin{align*}
	I(u)=\into \l( \frac{|\nabla u|^{p(x)}}{p(x)} 
	+ \mu(x) \frac{|\nabla u|^{q(x)}}{q(x)}\r)\diff x
\end{align*}
for all $u \in \Wpzero{\mathcal{H}}$.

\begin{proposition}\label{proposition-C1Energy}
	Let hypothesis \eqref{condition_1} be satisfied. Then the functional $I$ 
is well-defined and of class $C^1$ with $I'(u) = A(u)$.	
\end{proposition}

\begin{proof}
	The energy functional $I$ is well defined for any $u \in W_0^{\mathcal{H}}(\Omega)$ since
	\begin{equation*}
		0 \leq \frac{\rho_{\mathcal{H}}(\nabla u)}{q_+} \leq I(u) \leq \frac{\rho_{\mathcal{H}}(\nabla u)}{p_-} < \infty.
	\end{equation*}

	The Gateaux derivative is given by $A$ since for any $u , h \in W_0^{\mathcal{H}}(\Omega)$, some $t \in \R$ and some $\theta_{x,t} \in (0,1)$ given by the Mean Value Theorem, we have
	\begin{align*}
		\into \frac{\abs{\nabla u + t \nabla h}^{p(x)} - \abs{\nabla u}^{p(x)}}{t p(x)} \diff x 
		& = \into \abs{\nabla u + \theta_{x,t} t \nabla h}^{p(x)-2} \left( \nabla u + \theta_{x,t} t \nabla h \right) \cdot \nabla h \diff x \\
		& \xrightarrow{t \to 0} \into \abs{\nabla u }^{p(x)-2}  \nabla u  \cdot 
\nabla h \diff x.
	\end{align*}
	The limit follows from the Dominated Convergence Theorem, Proposition \ref{proposition_embeddings} \textnormal{(i)}, H\"older's inequality and Proposition \ref{proposition_1} \textnormal{(iii)} and \textnormal{(iv)}, as for $0 < \abs{t} < t_0$, it holds
	\begin{align*}
		& \abs{\nabla u + \theta_{x,t} t \nabla h}^{p(x)-2} \left( \nabla u + \theta_{x,t} t \nabla h \right) \cdot \nabla h
		\leq 2^{p_+ - 1} \left( \abs{\nabla u}^{p(x) -1} + t_0 \abs{\nabla h}^{p(x) -1} \right) \abs{ \nabla h }
	\end{align*}
	and
	\begin{align*}
		& \into \abs{\nabla u}^{p(x) -1} \abs{ \nabla h } \diff x 
		\leq 2 \norm{\abs{\nabla u}^{p(\cdot) -1}}_{\frac{p(\cdot)}{p(x) -1}} \norm{\nabla h}_{p(\cdot)} \\
		& \qquad \qquad \qquad  \qquad \quad \;\; \leq 2  \l(\varrho_{p(\cdot)} 
(\abs{\nabla u})\r) ^{ a }    \norm{\nabla h}_{p(\cdot)} 
		\leq 2  \norm{\nabla u}_{p(\cdot)}^{ b }    \norm{\nabla h}_{p(\cdot)},
	\end{align*}
	where $a,b>0$ are the exponents depending on the cases of Proposition \ref{proposition_1} \textnormal{(iii)} and \textnormal{(iv)}.
	The same arguments work on the terms with exponent $q(\cdot)$ by using Proposition \ref{proposition_embeddings} \textnormal{(vi)} and splitting $\mu(x) = \mu(x)^{\frac{1}{q(x)}} \mu(x)^{\frac{q(x)-1}{q(x)}}$ for using H\"older's inequality.
	
	The $C^1$-property follows since for any sequence $u_n \to u$ in $W_0^{\mathcal{H}}(\Omega)$ and $h \in W_0^{\mathcal{H}}(\Omega)$ with $\norm{h}_{1,\mathcal{H}} = 1$, we have by H\"older's inequality and Proposition 
\ref{proposition_embeddings} \textnormal{(i)}
	\begin{align*}
		& \into \left(  \abs{\nabla u_n }^{p(x)-2}  \nabla u_n - \abs{\nabla u }^{p(x)-2}  \nabla u \right)   \cdot \nabla h \diff x \\
		& \leq 2 \norm{\abs{\abs{\nabla u_n }^{p(\cdot)-2}  \nabla u_n - \abs{\nabla u }^{p(\cdot)-2}  \nabla u }}_{\frac{p(\cdot)}{p(\cdot)-1}}
		\norm{\nabla h}_{p(\cdot)} \xrightarrow{n \to \infty} 0.
	\end{align*}
	The convergence in $L^{\frac{p(\cdot)}{p(\cdot)-1}} (\Omega)$ follows by 
Proposition \ref{proposition_1} \textnormal{(v)} and Vitali's Theorem (see Bogachev \cite[Theorem 4.5.4]{Bogachev-2007}) since we have the convergence in measure because of $\nabla u_n \to \nabla u$ in $L^{p(\cdot)}(\Omega)$ and we have the uniform integrability by the uniform integrability of the sequence $\{\abs{\nabla u_n}^{p(x)}\}_{n\in\N}$ due to the same convergence in $L^{p(\cdot)}$ as before and 
	\begin{align*}
		\abs{\abs{\nabla u_n }^{p(x)-2}  \nabla u_n - \abs{\nabla u }^{p(x)-2}  
\nabla u }^{\frac{p(x)}{p(x)-1}}
		\leq 2^{\frac{p_+}{p_- -1} -1} \left( \abs{\nabla u_n }^{p(x)} + \abs{\nabla u }^{p(x)}  \right). 
	\end{align*}
	The same arguments work on the terms with exponents $q(x)$  by using Proposition \ref{proposition_1}\textnormal{(iii)}, \textnormal{(iv)}, Proposition \ref{proposition_embeddings}\textnormal{(vi)} and splitting $\mu(x) 
= \mu(x)^{\frac{1}{q(x)}} \mu(x)^{\frac{q(x)-1}{q(x)}}$ again for using 
H\"older's inequality.
\end{proof}

Before we give the main properties of the operator, we first state a general version of the reverse H\"older inequality. Since we did not find any 
reference, we will also give the proof for it.

\begin{lemma}\label{LemmaReverseHoelderVarExp}
	Let $(S,\Sigma,\lambda)$ be a measure space with $\lambda(S)>0$ and let $r\colon S \to [1,\infty)$ be measurable with $1<r_{-} := \essinf_S r \leq r_{+} := \esssup_S r < \infty$. Then for any measurable functions $f,g\colon S \to \mathbb{K}$  such that $g(s) \neq 0$ $\mu$-a.e. it holds
	\begin{align*}
		& \max \left\lbrace \norm{fg}_1^{\frac{1}{r_{-}}} , \norm{fg}_1^{\frac{1}{r_{+}}}  \right\rbrace \\
		& \geq \left[ \frac{1}{r_-} + \frac{1}{r'_-} \right]^{-1} 	\norm{\abs{f}^{\frac{1}{r(\cdot)}}}_1 \min \left\lbrace  \norm{\abs{g}^\frac{-1}{r(\cdot)-1}}_1^{ \frac{1- r_+}{r_-} } ,  \norm{\abs{g}^\frac{-1}{r(\cdot)-1}}_1^{ \frac{1- r_-}{r_+} } \right\rbrace \\
		& \geq \frac{1}{2} \norm{\abs{f}^{\frac{1}{r(\cdot)}}}_1 \min \left\lbrace  \norm{\abs{g}^\frac{-1}{r(\cdot)-1}}_1^{ \frac{1- r_+}{r_-} } , \norm{\abs{g}^\frac{-1}{r(\cdot)-1}}_1^{ \frac{1- r_-}{r_+} } \right\rbrace.
	\end{align*}
\end{lemma}

\begin{proof}
	Applying H\"older's inequality one gets
	\begin{equation*}
		\norm{\abs{f}^{\frac{1}{r(\cdot)}}}_1  = \norm{\abs{fg}^{\frac{1}{r(\cdot)}}\abs{g}^{\frac{-1}{r(\cdot)}}}_1
		\leq \left[ \frac{1}{r_-} + \frac{1}{r'_-} \right]
		\norm{\abs{fg}^{\frac{1}{r(\cdot)}}}_{r(\cdot)} \norm{\abs{g}^{\frac{-1}{r(\cdot)}}}_{\frac{r(\cdot)}{r(\cdot)-1}}.
	\end{equation*}
	From the comparison between the norm and the modular, see Proposition \ref{proposition_1} \textnormal{(iii)} and \textnormal{(iv)}, we obtain
	\begin{align*}
		\norm{\abs{f}^{\frac{1}{r(\cdot)}}}_1 
		& \leq \left[ \frac{1}{r_-} + \frac{1}{r'_-} \right] 
		\max \left\lbrace \norm{fg}_1^{\frac{1}{r_{-}}} , \norm{fg}_1^{\frac{1}{r_{+}}}  \right\rbrace \max \left\lbrace  \norm{\abs{g}^\frac{-1}{r(\cdot)-1}}_1^{ \frac{r_+-1}{r_-} } , 
		\norm{\abs{g}^\frac{-1}{r(\cdot)-1}}_1^{ \frac{r_- -1}{r_+} } \right\rbrace.		
	\end{align*}
\end{proof}

Now we are in the position to present the main properties of the operator 
$A$ motivated by the work of Liu-Dai \cite{Liu-Dai-2018} for the constant 
exponent case.

\begin{theorem}\label{prop_properties_operator}$~$
	\begin{enumerate}
		\item[\textnormal{(i)}]
			Let hypothesis \eqref{condition_1} be satisfied. Then the operator $A\colon \Wpzero{\mathcal{H}}\to \Wpzero{\mathcal{H}}^*$ is continuous, bounded and strictly monotone.
		\item[\textnormal{(ii)}]
			Let hypothesis \textnormal{(H)} be satisfied. Then the operator $A\colon \Wpzero{\mathcal{H}}\to \Wpzero{\mathcal{H}}^*$ satisfies the $\Ss$-property, that is,
			\begin{align*}
				u_n \weak u \text{ in }\Wpzero{\mathcal{H}}\quad\text{and}\quad \limsup_{n\to\infty} \langle A(u_n),u_n-u\rangle \leq 0,
			\end{align*}
			imply $u_n\to u$ in $\Wpzero{\mathcal{H}}$;
		\item[\textnormal{(iii)}]
			Let hypothesis \textnormal{(H)} be satisfied. Then the operator $A\colon \Wpzero{\mathcal{H}}\to \Wpzero{\mathcal{H}}^*$ is coercive and a homeomorphism.
	\end{enumerate}
\end{theorem}

\begin{proof}
	\textnormal{(i)} By Proposition \ref{proposition-C1Energy}, $A=I'$ with $I$ of class $C^1$, so $A$ is continuous.
	
	From the well-known inequality
	\begin{equation*}
		\l(|\xi|^{r-2}\xi -|\eta|^{r-2}\eta\r) \cdot (\xi-\eta)>0 \quad\text{if 
} r >1, \text{ for all }\xi,\eta \in \R^N \text{ with }\xi\neq\eta,
	\end{equation*}
	we see that
	\begin{align*}
		\l\langle A(u) - A (v) , u - v \r\rangle
		& = \int_\Omega \left( |\nabla u|^{p(x)-2}\nabla u - |\nabla v|^{p(x)-2}\nabla v \r)  \cdot \l( \nabla u - \nabla v \r)\diff x\\
		& \quad + \int_\Omega \mu(x) \l( |\nabla u|^{q(x)-2}\nabla u - |\nabla v|^{q(x)-2}\nabla v \r)  \cdot	\l( \nabla u - \nabla v \r)\diff x>0
	\end{align*}
	whenever $u\neq v$ which proves the strict monotonicity of $A$. Let us now prove that $A$ is bounded. Taking $u,v \in \Wpzero{\mathcal{H}}\setminus \{0\}$, by applying Young's inequality, we obtain
	\begin{align*}
		&\min\l\{\frac{1}{\|u\|_{1,\mathcal{H}}^{q_+-1}},\frac{1}{\|u\|_{1,\mathcal{H}}^{p_--1}}\r\} \l \lan A(u),\frac{v}{\|v\|_{1,\mathcal{H}}} \r\ran\\
		& \leq \min\l\{\frac{1}{\|u\|_{1,\mathcal{H}}^{q_+-1}},\frac{1}{\|u\|_{1,\mathcal{H}}^{p_--1}}\r\}\into \l[|\nabla u|^{p(x)-1} \frac{|\nabla v|}{\|v\|_{1,\mathcal{H}}} \mathop{+} \mu(x) |\nabla u|^{q(x)-1} \frac{|\nabla v|}{\|v\|_{1,\mathcal{H}}}\r]\diff x\\
		& \leq \into \l[\l|\frac{\nabla u}{\|u\|_{1,\mathcal{H}}}\r|^{p(x)-1} \frac{|\nabla v|}{\|v\|_{1,\mathcal{H}}}+ \mu(x)^{\frac{q(x)-1}{q(x)}} \l|\frac{\nabla u}{\|u\|_{1,\mathcal{H}}}\r|^{q(x)-1} \mu(x)^{\frac{1}{q(x)}}\frac{|\nabla v|}{\|v\|_{1,\mathcal{H}}}\r]\diff x\\
		&\leq \frac{p_+-1}{p_-} \into \l|\frac{\nabla u}{\|u\|_{1,\mathcal{H}}}\r|^{p(x)}\diff x
		+ \frac{1}{p_-} \into\l|\frac{\nabla v}{\|v\|_{1,\mathcal{H}}}\r|^{p(x)}\diff x\\
		& \quad +\frac{q_+-1}{q_-} \into \mu(x) \l|\frac{\nabla u}{\|u\|_{1,\mathcal{H}}}\r|^{q(x)}\diff x
		+ \frac{1}{q_-} \into \mu(x)\l|\frac{\nabla v}{\|v\|_{1,\mathcal{H}}}\r|^{q(x)}\diff x\\
		& \leq \frac{q_+-1}{p_-} \rho_{\mathcal{H}} \l(\frac{\nabla  u}{\| u\|_{1,\mathcal{H}}}\r)+\frac{1}{p_-} \rho_{\mathcal{H}} \l(\frac{\nabla v}{\|v\|_{1,\mathcal{H}}}\r)\\
		&\leq   \frac{q_+-1}{p_-} \rho_{\mathcal{H}} \l(\frac{\nabla  u}{\| u\|_{1,\mathcal{H},0}}\r)+\frac{1}{p_-} \rho_{\mathcal{H}} \l(\frac{\nabla v}{\|v\|_{1,\mathcal{H},0}}\r)
		=\frac{q_+-1}{p_-}+\frac{1}{p_-} = \frac{q_+}{p_-}.
	\end{align*}
	This fact gives
	\begin{align*}
		\|A(u)\|_*
		&= \sup_{\substack{{v \in W^{1,\mathcal{H}}_0(\Omega)} \\ {v \neq 0}}} \dfrac{\left\langle A(u) , v \right\rangle }{\|v\|_{1,\mathcal{H}}}
		\leq \frac{q_+}{p_-}
		\max \l\{\|u\|_{1,\mathcal{H}}^{q_+-1},\|u\|_{1,\mathcal{H}}^{p_--1}\r\}.
	\end{align*}
	Hence, $A$ is bounded.
	
	\textnormal{(ii)} Let $\{u_n\}_{n \in \N} \subseteq W^{1,\mathcal{H}}_0(\Omega)$ be a sequence such that
	\begin{align}\label{assumption_splus}
		u_n \weak u \text{ in }W^{1,\mathcal{H}}_0(\Omega)\quad \text{and}\quad 
\limsup_{n \to \infty}{\left\langle A(u_n), u_n - u \right\rangle } \leq 0.
	\end{align}
	The weak convergence of $u_n$ to $u$ in $\Wpzero{\mathcal{H}}$ yields
	\begin{align*}
		\lim_{n \to \infty} \left\langle A(u) , u_n - u \right\rangle = 0.
	\end{align*}
	This fact along with \eqref{assumption_splus} gives
	\begin{align*}
		\limsup_{n \to \infty}{\left\langle A(u_n)-A(u), u_n - u \right\rangle } \leq 0.
	\end{align*}
	Then, the strict monotonicity of $A$ implies that
	\begin{align*}
		0
		\leq \liminf_{n \to \infty} \left\langle A(u_n) - A(u) , u_n - u \right\rangle
		\leq \limsup_{n \to \infty}\left\langle A(u_n) - A(u) , u_n - u \right\rangle
		\leq 0.
	\end{align*}
	Hence, we get
	\begin{equation}\label{proof_2}
		\lim_{n \to \infty} \left\langle A(u_n) - A(u) , u_n - u \right\rangle  
= 0=\lim_{n \to \infty} \left\langle A(u) , u_n - u \right\rangle.
	\end{equation}

	{\bf Claim: } $\nabla u_n \to \nabla u$ in $L^{p(\cdot)}(\Omega)$

	Splitting the integral in $\{p \geq 2\}:=\{x\in\Omega\,:\, p(x)\geq 2\}$ and $\{p < 2\}:=\{x\in\Omega\,:\, p(x)<2\}$ and noting that all four 
terms in \eqref{proof_2} are non-negative yields
	\begin{align}\label{EqCasePgeq2}
		\lim_{n \to \infty} \int_{ \{p \geq 2 \}} 
		\left( \abs{\nabla u_n}^{p(x)-2} \nabla u_n - \abs{\nabla u}^{p(x)-2} \nabla u \right) \cdot \left( \nabla u_n - \nabla u \right) \diff x = 0, 
\\
		\label{EqCaseP<2}
		\lim_{n \to \infty} \int_{ \{p < 2 \}} 
		\left( \abs{\nabla u_n}^{p(x)-2} \nabla u_n - \abs{\nabla u}^{p(x)-2} \nabla u \right) \cdot \left( \nabla u_n - \nabla u \right) \diff x = 0.
	\end{align}

	From Simon \cite[formula (2.2)]{Simon-1978} we have the well-known  inequalities 
	\begin{align}
		c_p |\xi-\eta|^p & \leq \l(|\xi|^{p-2}\xi -|\eta|^{p-2}\eta\r) \cdot (\xi-\eta) \quad\text{if } p \geq 2,\label{IneqPgeq2}\\
		C_p |\xi-\eta|^2 & \leq \l(|\xi|^{p-2}\xi -|\eta|^{p-2}\eta\r) \cdot (\xi-\eta) \l(|\xi|^p+|\eta|^p\r)^{\frac{2-p}{p}}\quad\text{if } 1 \leq p \leq 2,\label{IneqPBetween1and2}
	\end{align}
	for all $\xi,\eta$ where
	\begin{align}\label{not_optimal_constants}
		c_p= 5^{\frac{2-p}{2}} \quad \text{and}\quad C_p= (p-1) 2^{\frac{(p-1)(p-2)}{p}},
	\end{align}
	see also Lindqvist \cite[chapter 12]{Lindqvist-2019}. Note that the constants in \eqref{not_optimal_constants} are not optimal, but sufficient for our treatment.

	From \eqref{EqCasePgeq2} and by the inequality \eqref{IneqPgeq2} it follows
	\begin{equation*}
		\lim_{n \to \infty} \int_{ \{p \geq 2 \}} \abs{\nabla u_n - \nabla u}^{p(x)} \diff x = 0,
	\end{equation*}
	and in the following lines it will be proved that the same holds in $\{p 
< 2\}$, hence 
	\begin{equation*}
		\lim_{n \to \infty} \varrho_{p(\cdot)} (\nabla u_n - \nabla u)
		= \lim_{n \to \infty} \int_{\Omega}
		\abs{\nabla u_n - \nabla u}^{p(x)} \diff x = 0,
	\end{equation*}
	so the Claim follows by Proposition \ref{proposition_1}.
	
	Let $E_n = \{ \nabla u_n \neq 0\} \cup \{ \nabla u \neq 0\}$. By the absolute continuity of the Lebesgue integral, as the integrand is zero outside $E_n$, and by \eqref{IneqPBetween1and2} with $p_{+,k} = 2 - 1/k$ (note also $p-2 < 0$) it follows
	\begin{align*}
		& \int_{\{p<2\}} \left( \abs{\nabla u_n}^{p(x)-2} \nabla u_n - \abs{\nabla u}^{p(x)-2} \nabla u \right) \cdot \left( \nabla u_n - \nabla u \right) \diff x \\
		& = \lim_{k \to \infty} \int_{ E_n \cap \{ p \leq p_{+,k}\} } \left( \abs{\nabla u_n}^{p(x)-2} \nabla u_n - \abs{\nabla u}^{p(x)-2} \nabla u \right) \cdot \left( \nabla u_n - \nabla u \right) \diff x \\
		& \geq \limsup_{k \to \infty} \ (p_- -1) 2^{\frac{(p_{+,k} -1)(p_- -2)}{p_-}}\int_{E_n \cap \{p \leq p_{+,k} \}} \abs{\nabla u_n - \nabla u}^2 \left( \abs{\nabla u_n}^{p(x)} + \abs{\nabla u}^{p(x)} \right) ^{\frac{p(x)-2}{p(x)}} \diff x \\
		& \geq (p_- -1) 2^{\frac{(p_- -2)}{p_-}} \limsup_{k \to \infty}  \int_{E_n \cap \{p \leq 2 - 1/k \}} \abs{\nabla u_n - \nabla u}^2 \left( \abs{\nabla u_n}^{p(x)} + \abs{\nabla u}^{p(x)} \right) ^{\frac{p(x)-2}{p(x)}} \diff x. 
	\end{align*}
	By \eqref{EqCaseP<2}, for $n\geq n_0$ for some $n_0 \in \mathbb{N}$, the 
limit superior is strictly smaller than one, thus the same holds for the integrals $k$ large enough. Hence, when we apply Lemma \ref{LemmaReverseHoelderVarExp} to our integral with $r(\cdot)=\frac{2}{p(\cdot)}$ the maximum of the left-hand side of the lemma is attained at $\norm{fg}_1^{\frac{1}{r_+}}$. Note that (consider $p_{-,k} = p_-$)
	\begin{equation*}
		r_{\pm,k} = \frac{2}{p_{\mp,k}} \qquad \text{ and } \qquad  \frac{1 - 
r_{\pm,k} }{r_{\mp,k}} = \frac{p_{\pm,k} (p_{\mp,k} - 2)}{2 p_{\mp,k}}.
	\end{equation*}
	Applying this result, again because the integrands are zero outside $E_n$, as $\{u_n\}_{n \in \N}$ is bounded in the modular by some constant $M>1$ (due to its weak convergence and Proposition \ref{proposition_1} \textnormal{(iii)} and \textnormal{(iv)}) and as $(p_{\pm,k} -2) < 0$, for $n \geq n_0$, we have
	\begin{align*}
		& \int_{\{p<2\}} \left( \abs{\nabla u_n}^{p(x)-2} \nabla u_n - \abs{\nabla u}^{p(x)-2} \nabla u \right) \cdot \left( \nabla u_n - \nabla u \right) \diff x \\
		& \geq (p_- -1) 2^{\frac{(p_- -2)}{p_-}} \limsup_{k \to \infty} \frac{1}{2^{\frac{2}{p_-}}} 
		\quad \cdot \left( \int_{\{p < 2 - 1/k \}} \abs{\nabla u_n - \nabla u}^{p(x)} \diff x \right) ^{\frac{2}{p_-}} \\
		& \quad \times  \min \left\lbrace  \left(  \int_{\{p < 2 - 1/k\}} \left( \abs{\nabla u_n}^{p(x)} + \abs{\nabla u}^{p(x)} \right) \diff x \right) 
^{\frac{p_{+,k}(p_- -2)}{p_-^2}},\right.\\
		&\left.\qquad \qquad\quad
		\left(  \int_{\{ p < 2 - 1/k \}} \left( \abs{\nabla u_n}^{p(x)} + \abs{\nabla u}^{p(x)} \right) \diff x \right) ^{\frac{(p_{+,k} -2)}{p_{+,k}}} \right\rbrace \\
		& \geq (p_- -1) 2^{\frac{(p_- -4)}{p_-}} 
		\limsup_{k \to \infty} 
		\left( \int_{\{p < 2 - 1/k \}} \abs{\nabla u_n - \nabla u}^{p(x)} \diff 
x \right) ^{\frac{2}{p_-}} \\
		& \qquad \times \min \left\lbrace \left( M + \int_{\Omega} \abs{\nabla u}^{p(x)} \diff x  \right)^{\frac{p_{+,k}(p_- -2)}{p_-^2}} ,
		\left( M + \int_{\Omega} \abs{\nabla u}^{p(x)} \diff x \right)^{\frac{(p_{+,k} -2)}{p_{+,k}}} \right\rbrace \\
		& \geq K \limsup_{k \to \infty} 
		\left( \int_{\{p < 2 - 1/k \}} \abs{\nabla u_n - \nabla u}^{p(x)} \diff 
x \right) ^{\frac{2}{p_-}} \\
		& = K \left( \int_{\{p < 2 \}} \abs{\nabla u_n - \nabla u}^{p(x)} \diff x \right) ^{\frac{2}{p_-}},
	\end{align*}
	where
	\begin{align*}
		K = & (p_- -1) 2^{\frac{(p_- -4)}{p_-}} \min \left\lbrace \left( M + \int_{\Omega} \abs{\nabla u}^{p(x)} \diff x \right)^{\frac{2(p_- -2)}{p_-^2}} ,
		\left( M + \int_{\Omega} \abs{\nabla u}^{p(x)} \diff x \right)^{\frac{(p_- -2)}{p_-}}  \right\rbrace.
	\end{align*}
	By \eqref{EqCaseP<2} it follows that
	\begin{equation*}
		\lim_{n \to \infty} \int_{\{p < 2\}} \abs{\nabla u_n - \nabla u}^{p(x)} 
\diff x = 0.
	\end{equation*} 
	This proves the Claim.

	From the Claim  we know that $\{\nabla u_n\}_{n \in \N}$ converges in measure to $\nabla u$ in $\Omega$. Applying Young's inequality gives
	\begin{align*}
		\begin{split}
			& \int_\Omega \left(  \abs{\nabla u_n}^{p(x)-2}\nabla u_n 
			+ \mu(x) \abs{\nabla u_n}^{q(x)-2}\nabla u_n \right)  \cdot \left( \nabla u_n - \nabla u \right)  \diff x\\
			& = \int_\Omega |\nabla u_n|^{p(x)}\diff x- \int_\Omega |\nabla u_n|^{p(x)-2}\nabla u_n \cdot \nabla u\diff x\\
			& \quad + \int_\Omega \mu(x)|\nabla u_n|^{q(x)}\diff x- \int_\Omega \mu(x)|\nabla u_n|^{q(x)-2}\nabla u_n \cdot \nabla u\diff x\\
			& \geq \int_\Omega |\nabla u_n|^{p(x)}\diff x- \int_\Omega |\nabla u_n|^{p(x)-1} |\nabla u|\diff x\\
			& \quad + \int_\Omega \mu(x)|\nabla u_n|^{q(x)}\diff x- \int_\Omega \mu(x)|\nabla u_n|^{q(x)-1}|\nabla u|\diff x\\
			& \geq \int_\Omega |\nabla u_n|^{p(x)}\diff x- \int_\Omega \l( \frac{p(x)-1}{p(x)}|\nabla u_n|^{p(x)}+\frac{1}{p(x)} |\nabla u|^{p(x)}\r)\diff x\\
			& \quad + \int_\Omega \mu(x)|\nabla u_n|^{q(x)}\diff x- \int_\Omega \mu(x)\l(\frac{q(x)-1}{q(x)}|\nabla u_n|^{q(x)}+\frac{1}{q(x)}|\nabla u|^{q(x)}\r)\diff x\\
			& = \int_\Omega \frac{1}{p(x)}|\nabla u_n|^{p(x)}\diff x- \int_\Omega \frac{1}{p(x)} |\nabla u|^{p(x)}\diff x\\
			& \quad + \int_\Omega \frac{\mu(x)}{q(x)}|\nabla u_n|^{q(x)}\diff x- \int_\Omega \frac{\mu(x)}{q(x)}|\nabla u|^{q(x)}\diff x .
		\end{split}
	\end{align*}
	Hence by \eqref{proof_2}
	\begin{align}\label{proof_3}
		\limsup_{n \to \infty} \int_\Omega \l(\dfrac{\abs{\nabla u_n}^{p(x)}}{p(x)} + \mu(x) \dfrac{ \abs{\nabla u_n}^{q(x)}}{q(x)}\r) \diff x 
		\leq \int_\Omega \l(\dfrac{\abs{\nabla u}^{p(x)}}{p(x)} + \mu(x) \dfrac{ \abs{\nabla u}^{q(x)}}{q(x)}\r) \diff x.
	\end{align}
	From Fatou's Lemma, we obtain
	\begin{align}\label{proof_1}
		\liminf_{n \to \infty} \int_\Omega \l(\frac{|\nabla u_n|^{p(x)}}{p(x)} + \mu(x) \frac{|\nabla u_n|^{q(x)}}{q(x)}\r) \diff x
		\geq \int_\Omega \l(\frac{|\nabla u|^{p(x)}}{p(x)} + \mu(x) \frac{|\nabla u|^{q(x)}}{q(x)}\r) \diff x.
	\end{align}
	Combining \eqref{proof_1} and \eqref{proof_3} we conclude that
	\begin{align}\label{proof_4}
		\lim_{n \to \infty} \int_\Omega \l(\frac{|\nabla u_n|^{p(x)}}{p(x)} + \mu(x) \frac{|\nabla u_n|^{q(x)}}{q(x)}\r) \diff x
		= \int_\Omega \l(\frac{|\nabla u|^{p(x)}}{p(x)} + \mu(x) \frac{|\nabla u|^{q(x)}}{q(x)}\r) \diff x.
	\end{align}
	
	By the Claim we have that $\{\nabla u_n\}_{n \in \N}$ converges in measure to $\nabla u$, so by straightforward computations the functions on the 
left-hand side of \eqref{proof_4} converge in measure to those on the right-hand side. The converse of Vitali's theorem (see Bauer \cite[Lemma 21.6]{Bauer-2001}) yields the uniform integrability of the sequence of functions
	\begin{align*}
		\l\{\frac{|\nabla u_n|^{p(x)}}{p(x)}+ \mu(x) \frac{|\nabla u_n|^{q(x)}}{q(x)}\r\}_{n\in\N}.
	\end{align*}
	
	On the other side, we know that
	\begin{align*}
		\begin{split}
			&|\nabla u_n-\nabla u|^{p(x)}+\mu(x)|\nabla u_n-\nabla u|^{q(x)}\\
			& \leq 2^{q_+ -1} q_+ \l(\frac{|\nabla u_n|^{p(x)}}{p(x)} + \mu(x) \frac{|\nabla u_n|^{q(x)}}{q(x)} + \frac{|\nabla u|^{p(x)}}{p(x)} + \mu(x) \frac{|\nabla u|^{q(x)}}{q(x)} \r),
		\end{split}
	\end{align*}
	which implies that the sequence of functions
	\begin{align*}
		\l\{|\nabla u_n-\nabla u|^{p(x)}+\mu(x)|\nabla u_n-\nabla u|^{q(x)}\r\}_{n\in\N}
	\end{align*}
	is uniformly integrable. By straightforward computations and using the convergence in measure of $\left\lbrace \nabla u_n \right\rbrace _{n \in \mathbb{N}}$ to $\nabla u$, this sequence converges in measure to 0. Applying Vitali's theorem (see Bogachev \cite[Theorem 4.5.4]{Bogachev-2007}) it follows that
	\begin{align*}
		& \lim_{n \to \infty} \rho_{\mathcal{H}} ( \nabla u_n - \nabla u )= 
\lim_{n \to \infty} \into  \l(|\nabla u_n - \nabla u|^{p(x)} + \mu(x) |\nabla u_n - \nabla u|^{q(x)}\r) \diff x
	= 0,
	\end{align*}
	which is equivalent to $\norm{u_n - u}_{1,\mathcal{H},0} \to 0$, see Proposition \ref{proposition_modular_properties} \textnormal{(v)}. Hence, $u_n\to u$ in $\Wpzero{\mathcal{H}}$.
	
	\textnormal{(iii)} The operator $A$ is coercive, since by Proposition \ref{proposition_modular_properties} \textnormal{(i)}, one has
	\begin{align*}
		\frac{\left\langle A(u) , u \right\rangle}{\norm{u}_{1,\mathcal{H},0}} 
		& = \int_\Omega \l(\norm{u}^{p(x)-1}_{1,\mathcal{H},0} \left( \dfrac{\abs{\nabla u}}{\norm{u}_{1,\mathcal{H},0}} \right)^{p(x)} 
		+ \norm{u}^{q(x)-1}_{1,\mathcal{H},0} \mu(x) \left( \dfrac{\abs{\nabla u}}{\norm{u}_{1,\mathcal{H},0}} \right)^{q(x)}\r) \diff x \\
		& \geq \min\left\lbrace \norm{u}^{p_- - 1}_{1,\mathcal{H},0} , \norm{u}^{q_+ -1}_{1,\mathcal{H},0} \right\rbrace 
		\rho_\mathcal{H} \left( \dfrac{\nabla u}{\norm{\nabla u}_\mathcal{H}} \right) \to + \infty \text{ as } \norm{u}_{1,\mathcal{H},0} \to + \infty.
	\end{align*}
	This fact along with assertion \textnormal{(i)} of this theorem implies by the Minty-Browder theorem, see, for example, Zeidler \cite[Theorem 26.A]{Zeidler-1990}, that $A$ is invertible and that $A^{-1}$ is strictly monotone, demicontinuous and bounded. We only need to show that $A^{-1}$ is 
continuous.

	To this end, let $\{y_n\}_{n \in \N} \subseteq \Wpzero{\mathcal{H}}^*$ be a sequence such that $y_n \to y$ in $\Wpzero{\mathcal{H}}^*$ and let $u_n=A^{-1}(y_n)$ as well as $u=A^{-1}(y)$. By the strong convergence of $\{y_n\}_{n \in \N}$ and the boundedness of $A^{-1}$ we know that $u_n$ 
is bounded in $\Wpzero{\mathcal{H}}$. Hence, there exists a subsequence $\{u_{n_k}\}_{k\in\N}$ of $\{u_n\}_{n\in\N}$ such that
	\begin{align*}
		u_{n_k} \rightharpoonup u_0 \quad\text{in }\Wpzero{\mathcal{H}}.
	\end{align*}
	Using these facts we have
	\begin{align*}
		&\lim_{k\to \infty} \l\lan A(u_{n_k})-A(u_0),u_{n_k}-u_0\r\ran\\
		&= \lim_{k \to \infty} \l\lan y_{n_k}-y,u_{n_k}-u_0\r\ran+\lim_{k \to 
\infty} \l\lan y-A(u_0),u_{n_k}-u_0\r\ran=0.
	\end{align*}
	From assertion \textnormal{(ii)} of the theorem we know that $A$ fulfills the $\Ss$-property which implies that $u_{n_k}\to u_0$ in $\Wpzero{\mathcal{H}}$. By the continuity of the operator $A$ we easily see that
	\begin{align*}
		A(u_0)=\lim_{k \to \infty} A(u_{n_k})=\lim_{k \to \infty} y_{n_k}=y=A(u).
	\end{align*}
	Since $A$ is injective, it follows that $u=u_0$. By the subsequence principle it is easy to show that the whole sequence converges.
\end{proof}

We have similar results when the operator $A$ acts on $\Wp{\mathcal{H}} $.

\begin{proposition}$~$
	\begin{enumerate}
		\item[\textnormal{(i)}]
		Let hypothesis \eqref{condition_1} be satisfied. Then the functional $I\colon \Wp{\mathcal{H}} \to \R$ is well-defined and of class $C^1$ with $I'(u) = A(u)$.
		\item[\textnormal{(ii)}]
		Let hypothesis \eqref{condition_1} be satisfied. Then the operator $A\colon \Wp{\mathcal{H}}\to \Wp{\mathcal{H}}^*$ is continuous, bounded and strictly monotone.
		\item[\textnormal{(iii)}]
		Let hypothesis \textnormal{(H)} be satisfied. Then the operator $A\colon \Wp{\mathcal{H}}\to \Wp{\mathcal{H}}^*$ satisfies the $\Ss$-property, that is,
		\begin{align*}
			u_n \weak u \text{ in }\Wp{\mathcal{H}}\quad\text{and}\quad \limsup_{n\to\infty} \langle A(u_n),u_n-u\rangle \leq 0,
		\end{align*}
		imply $u_n\to u$ in $\Wp{\mathcal{H}}$.
	\end{enumerate}
\end{proposition}

\begin{proof}
	The assertions \textnormal{(i)} and \textnormal{(ii)} follow in the same 
way as in the proof of Theorem \ref{prop_properties_operator}. For \textnormal{(iii)} we make use of the compact embedding $\Wp{\mathcal{H}}\hookrightarrow \Lp{\mathcal{H}}$, see Proposition \ref{prop-generalpoincare} \textnormal{(i)}.
\end{proof}

In the following let $X=\Wpzero{\mathcal{H}}$ or $X=\Wp{\mathcal{H}}$ 
and let $B\colon X\to X^*$ be given by
\begin{align*}
	\l\lan B(u),v \r\ran_{X}
	&=\into \l( |\nabla u|^{p(x)-2}\nabla u \cdot \nabla v +\mu(x) |\nabla 
u|^{q(x)-2}\nabla u \cdot \nabla v\r)\diff x\\
	&\quad +\into \l( |u|^{p(x)-2} u  v +\mu(x) | u|^{q(x)-2} u   v\r)\diff x,
\end{align*}
for all $u,v \in X$, where $\lan\cdot\,,\cdot\ran_{X}$ denotes the duality pairing between $X$ and its dual space $X^*$. Moreover, let $J\colon  X\to\R$ be given  by
\begin{align*}
	J(u)=\into \l( \frac{|\nabla u|^{p(x)}}{p(x)} 
	+ \mu(x) \frac{|\nabla u|^{q(x)}}{q(x)}\r)\diff x
	+\into \l( \frac{| u|^{p(x)}}{p(x)} 
	+ \mu(x) \frac{| u|^{q(x)}}{q(x)}\r)\diff x
\end{align*}
for all $u \in X$.

Under the weaker assumptions in \eqref{condition_1} we have the following.

\begin{proposition}
	Let hypothesis \eqref{condition_1} be satisfied.
	\begin{enumerate}
		\item[\textnormal{(i)}]
		The functional $J\colon X \to \R$ is well-defined and of class $C^1$ with $J'(u) = B(u)$.
		\item[\textnormal{(ii)}]
		The operator $B\colon X\to X^*$ is continuous, bounded and strictly monotone.
		\item[\textnormal{(iii)}]
		The operator $B\colon X\to X^*$ satisfies the $\Ss$-property, that is,
		\begin{align*}
			u_n \weak u \text{ in }X\quad\text{and}\quad \limsup_{n\to\infty} \langle B(u_n),u_n-u\rangle \leq 0,
		\end{align*}
		imply $u_n\to u$ in $X$.
		\item[\textnormal{(iv)}]
		The operator $B\colon X\to X^*$ is coercive and a homeomorphism.
	\end{enumerate}
\end{proposition}

\section{Existence and uniqueness results}\label{section_4}

In this section we prove our main existence and uniqueness results. Recall that the problem under consideration is the following one
\begin{equation}\label{problem}
	\begin{aligned}
		-\divergenz\left(|\nabla u|^{p(x)-2}\nabla u+\mu(x) |\nabla u|^{q(x)-2}\nabla u\right) & =f(x,u,\nabla u)\quad && \text{in } \Omega,\\
		u & = 0 &&\text{on } \partial \Omega.
	\end{aligned}
\end{equation}

We suppose the following assumptions on the nonlinearity $f$.

\begin{enumerate}[leftmargin=1cm]
	\item[\textnormal{H(f)}]
		Let $f\colon \Omega \times \R \times \R^N \to \R$ be a Carath\'{e}odory 
function such that $f(\cdot,0,0)\neq 0$ and the following hold:
		\begin{enumerate}[leftmargin=0.5cm]
			\item[\textnormal{(i)}]
				There exists $\alpha \in \Lp{\frac{r(\cdot)}{r(\cdot)-1}}$ and $a_1, a_2\geq 0$ such that
				\begin{align*}
					|f(x,s,\xi)| & \leq a_1 |\xi|^{p(x) \frac{r(x)-1}{r(x)}}+a_2|s|^{r(x)-1}+\alpha(x)
				\end{align*}
				for a.\,a.\,$x\in\Omega$, for all $s\in \R$ and for all $\xi\in\R^N$, 
where $r \in C_+(\close)$ is such that $r(x)<p^*(x)$ for all $x\in \close$ with the critical exponent $p^*$ given in \eqref{critical_exponent}.
			\item[\textnormal{(ii)}]
				There exists $\omega \in \Lp{1}$ and $b_1,b_2\geq 0$ such that
				\begin{align}\label{estimate_f}
					f(x,s,\xi)s & \leq b_1|\xi|^{p(x)}+b_2|s|^{p_-}+\omega(x)
				\end{align}
				for a.\,a.\,$x\in\Omega$, for all $s\in \R$ and for all $\xi\in\R^N$. 
Moreover,
				\begin{align}\label{condition_coeffizients}
					1-b_1-b_2\lambda_{1,p_-}^{-1}>0
				\end{align}
				where $\lambda_{1,p_-}$ is the first eigenvalue of the Dirichlet eigenvalue problem for the $p_-$-Laplacian, see \eqref{eigenvalue_problem}.
		\end{enumerate}
\end{enumerate}

\begin{example}
	The following function satisfies hypotheses \textnormal{H(f)}:
	\begin{align*}
		f(x,s,\xi)=-d_1|s|^{r(x)-2}s+d_2|\xi|^{(p_- -1) \left( \frac{r(x) - 1}{r(x)} \right) } + d_3 \gamma(x), 
	\end{align*}
	with $l, r \in C_+(\close)$, $r(x)<p^*(x)$ and $l(x) \leq \min \{ p_- , r(x) \}$ for all $x \in \close$, $0 \neq \gamma \in \Lp{\frac{l(\cdot)}{l(\cdot)-1}}$, $d_1 \geq 0$ and
	\begin{align*}
		& 0 < |d_3| < p_- \lambda_{1,p_-} 
		\quad\text{as well as}\quad
	 	0 \leq |d_2| < \frac{p_- - |d_3| \lambda_{1,p_-}^{-1} }{p_- -1+\lambda_{1,p_-}^{-1}}.
	\end{align*}
\end{example}

\begin{definition}
	We say that $u\in \Wpzero{\mathcal{H}}$ is a weak solution of problem \eqref{problem} if for all test functions $\ph \in \Wpzero{\mathcal{H}}$ it 
satisfies
	\begin{align}\label{weak_solution}
		\into \left(|\nabla u|^{p(x)-2}\nabla u+\mu(x)|\nabla u|^{q(x)-2}\nabla 
u \right)\cdot\nabla\ph \diff x=\into f(x,u,\nabla u)\ph \diff x.
	\end{align}
\end{definition}
Because of Proposition \ref{proposition_embeddings} and hypothesis \textnormal{H(f)(i)} a weak solution in \eqref{weak_solution} is well-defined.

The following proposition is an immediate consequence of Theorem \ref{prop_properties_operator}.

\begin{proposition}
	Let hypothesis \textnormal{(H)} be satisfied and let
	\begin{equation*}
		f(x,s,\xi) = \alpha (x) \text{ for all } (x,s,\xi) \in \Omega \times \R \times \RN,
	\end{equation*} 
	where $r$ and $\alpha$ are as in \textnormal{H(f)(i)}. Then \eqref{problem} has a unique weak solution.
\end{proposition}

\begin{proof}
	By assumption and Proposition \ref{proposition_embeddings}, $\Wpzero{\mathcal{H}} \hookrightarrow \Lp{r(\cdot)}$, hence $\alpha \in \Lp{r'(\cdot)} = \Lp{r(\cdot)} ^* \hookrightarrow \Wpzero{\mathcal{H}}^*$ and by Proposition \ref{prop_properties_operator} $A$ is bijective.
\end{proof}

Our main existence result reads as follows.

\begin{theorem}\label{theorem_existence}
	Let hypotheses \textnormal{(H)} and \textnormal{H(f)} be satisfied. Then 
problem \eqref{problem} admits at least one nontrivial weak solution $u \in \Wpzero{\mathcal{H}}$.
\end{theorem}

\begin{proof}
	First, we introduce the Nemytskij operator related to $f$, that is, $N_f 
:=i^*\circ \hat N_f$, where $\hat N_f\colon \Wpzero{\mathcal{H}} \to \Lp{r'(\cdot)}$ is given by
	\begin{equation*}
		\hat{N}_f (u) = f(x,u,\nabla u),
	\end{equation*}
	and $i^*\colon\Lp{r'(\cdot)}\to \Wpzero{\mathcal{H}}^*$ is the adjoint operator of the embedding $i\colon\Wpzero{\mathcal{H}}\to \Lp{r(\cdot)}$. It is clear that $\hat{N}_f$ is well-defined, bounded and continuous by \textnormal{H(f)(i)} and Proposition \ref{proposition_embeddings} (for the 
continuity use Vitali's Theorem as in Proposition \ref{proposition-C1Energy}).

	 For $u \in \Wpzero{\mathcal{H}}$ we define $\mathcal{A}: \Wpzero{\mathcal{H}} \to \Wpzero{\mathcal{H}}^* $ by
	\begin{align*}
		\mathcal{A}(u)=A(u)-N_f(u),
	\end{align*}
	which consequently is continuous and bounded by Theorem \ref{prop_properties_operator}.
	 
	In order to apply Theorem \ref{theorem_pseudomonotone} we first show that $\mathcal{A}$ is pseudomonotone in the sense of Remark \ref{RemarkPseudomonotone}. 
   	To this end, let $\{u_n\}_{n\geq 1}\subseteq \Wpzero{\mathcal{H}}$ be 
a sequence such that
	\begin{align}\label{assumption_pseudomonotone}
		u_n \weak u\quad \text{in }\Wpzero{\mathcal{H}}\quad\text{and} \quad \limsup_{n\to \infty} \langle \mathcal{A}(u_n),u_n-u\rangle_{\mathcal{H}} \leq 0.
	\end{align}
	The compact embedding from Proposition \ref{embedding_critical} implies that
	\begin{align}\label{a1}
		u_n \to u\quad \text{in }\Lp{r(\cdot)}
	\end{align}
	since $r(x)<p^*(x)$ for all $x \in \close$. From H\"older's inequality, the weak convergence of $\{u_n\}_{n\in\N}$ in $\Wpzero{\mathcal{H}}$ (hence it is bounded in its norm) and the boundedness of $\hat{N}_f$ it follows that
	\begin{align*}
		\l|\into f\l(x,u_n,\nabla u_n\r)(u_n-u)\diff x \r|
		& \leq 2 \norm{\hat{N}_f(u_n)}_{\frac{r(\cdot) - 1}{r(\cdot)}} \norm{u-u_n}_{r(\cdot)} \\
		& \leq 2 \sup_{n \in \N} \norm{\hat{N}_f(u_n)}_{\frac{r(\cdot) - 1}{r(\cdot)}} \norm{u-u_n}_{r(\cdot)}.
	\end{align*}
	From this along with the strong convergence in \eqref{a1} we see that
	\begin{align*}
		\lim_{n\to \infty} \into f(x,u_n,\nabla u_n) (u_n-u)\diff x&=0.
	\end{align*}
	Hence, if we pass to the limit in the weak formulation in \eqref{weak_solution} replacing $u$ by $u_n$ and $\ph$ by $u_n-u$, we obtain
	\begin{align}\label{a4}
		\limsup_{n\to\infty} \langle A(u_n),u_n-u\rangle_{\mathcal{H}}=\limsup_{n\to \infty} \langle \mathcal{A}(u_n),u_n-u\rangle_{\mathcal{H}}\leq 0.
	\end{align}
	Since $A$ fulfills the $\Ss$-property, see Theorem \ref{prop_properties_operator}, by using \eqref{assumption_pseudomonotone} and \eqref{a4} it follows that $u_n\to u$ in $\Wpzero{\mathcal{H}}$. Therefore, by the continuity of the operator $\mathcal{A}$, we conclude that $\mathcal{A}(u_n)\to \mathcal{A}(u)$ in $\Wpzero{\mathcal{H}}^*$. Hence $\mathcal{A}$ is pseudomonotone.

	Let us now prove that $\mathcal{A}$ is coercive, see Definition \ref{SplusPM}. Recall the representation of the first eigenvalue of the $p_-$-Laplacian, see \eqref{lambda_one}, replacing $r$ by $p_-$, we have the inequality
	\begin{align}\label{estimate_lambda}
		\|u\|^{p_-}_{p_-} \leq \lambda_{1,p_-}^{-1} \|\nabla u\|_{p_-}^{p_-} \quad\text{for all }u\in \Wpzero{p_-}.
	\end{align}
	Note that $\Wpzero{\mathcal{H}}\subseteq \Wpzero{p_-}$. Then, by applying  \eqref{estimate_lambda} and \eqref{estimate_f} along with Proposition \ref{proposition_modular_properties} one has
	\begin{align*}
		\langle \mathcal{A}(u),u\rangle
		& = \into \left(|\nabla u|^{p(x)-2}\nabla u+\mu(x)|\nabla u|^{q(x)-2}\nabla u \right)\cdot\nabla u \diff x-\into f(x,u,\nabla u)u \diff x\\
		&\geq \rho_{\mathcal{H}}(\nabla u) -b_1\into |\nabla u|^{p(x)}\diff x-b_2\into |u|^{p_-}\diff x-\|\omega\|_1\\
		&\geq \rho_{\mathcal{H}}(\nabla u) -b_1\into |\nabla u|^{p(x)}\diff x-b_2 \lambda_{1,p_-}^{-1}\into |\nabla u|^{p_-}\diff x-\|\omega\|_1\\
		& \geq \left(1-b_1-b_2\lambda_{1,p_-}^{-1}\right) \rho_{\mathcal{H}}(\nabla u)-b_2 \lambda_{1,p_-}^{-1}|\Omega|-\|\omega\|_1\\
		& \geq \left(1-b_1-b_2\lambda_{1,p_-}^{-1}\right) \min \l\{\|\nabla u\|_{\mathcal{H}}^{q_+},\|\nabla u\|_{\mathcal{H}}^{p_-}\r\} -b_2 \lambda_{1,p_-}^{-1}|\Omega|-\|\omega\|_1.
	\end{align*}
	Hence, since $1<p_-<q_+$ and $1-b_1-b_2\lambda_{1,p_-}^{-1}>0$ by assumption \eqref{condition_coeffizients},  we conclude that the operator $\mathcal{A}\colon\Wpzero{\mathcal{H}}\to \Wpzero{\mathcal{H}}^*$ is coercive.

	Therefore, we have proved that the operator $\mathcal{A}\colon\Wpzero{\mathcal{H}}\to\Wpzero{\mathcal{H}}^*$ is bounded, pseudomonotone and coercive. Applying Theorem \ref{theorem_pseudomonotone} we get a function $u\in \Wpzero{\mathcal{H}}$ such that $\mathcal{A}(u)=0$. By the definition 
of the operator $\mathcal{A}$ and the first condition in \textnormal{H(f)}, $u$ is a nontrivial weak solution of our original problem \eqref{problem}. This completes the proof.
\end{proof}

In the second part of this section we want to discuss the question under what conditions the solution obtained in Theorem \ref{theorem_existence} is unique. In order to give a positive answer we need to strengthen our conditions on the nonlinearity $f\colon\Omega\times\R\times\R^N\to\R$ in the following sense.
\begin{enumerate}[leftmargin=1cm]
	\item[\textnormal{(U1)}]
		There exists $c_1\geq 0$ such that
		\begin{align*}
			&(f(x,s,\xi)-f(x,t,\xi))(s-t) \leq c_1 |s-t|^2
		\end{align*}
		for a.\,a.\,$x\in\Omega$, for all $s,t \in\R$ and for all $\xi\in \R^N$.
	\item[\textnormal{(U2)}]
		There exists $\rho\in \Lp{r'(\cdot)}$, where $r \in C_+(\close)$ is such that $r(x)<p^*(x)$ for all $x\in \close$, and $c_2\geq 0$ such that the 
map $\xi\mapsto f(x,s,\xi)-\rho(x)$ is linear for a.\,a.\,$x\in\Omega$, for all $s\in \R$ and
		\begin{align*}
			|f(x,s,\xi)-\rho(x)|\leq c_2|\xi|
		\end{align*}
		for a.\,a.\,$x\in\Omega$, for all $s \in\R$ and for all $\xi\in \R^N$. Moreover,
		\begin{align}\label{condition_coeffizients2}
			c_1\lambda_{1,2}^{-1}+c_2\lambda_{1,2}^{-\frac{1}{2}}<1,
		\end{align}
		where $\lambda_{1,2}$ is the first eigenvalue of the Dirichlet eigenvalue problem for the Laplace differential operator, see \eqref{eigenvalue_problem}.
\end{enumerate}

\begin{example}
	The following function satisfies hypotheses \textnormal{H(f)}, \textnormal{(U1)} and \textnormal{(U2)}, where for simplicity we drop the $s$-dependence:
	\begin{align*}
		f(x,\xi)=\sum_{i=1}^N \beta_i \xi_i+\rho(x)\quad\text{for a.\,a.\,} 
x\in\Omega \text{ and for all }\xi \in\R^N,
	\end{align*}
	with $ p_- = 2 $ , $0 \neq \rho \in \Lp{2}$ and
	\begin{align*}
		\beta =(\beta_1, \ldots, \beta_N)\in \RN \text{ with } | \beta | ^2 
		< \min\left\{ 1 - \frac{1}{2}\lambda_{1,2}^{-1} \ , \ \lambda_{1,2}\right\}.
	\end{align*}
	Any $r \in C_+(\close)$ such that $p_- = 2 \leq r(x) < p^*(x)$ for all 
$x \in \close$ is admissible.
\end{example}

Our uniqueness result reads as follows.

\begin{theorem}\label{theorem_uniqueness}
	Let \textnormal{(H)}, \textnormal{H(f)}, \textnormal{(U1)}, and \textnormal{(U2)} be satisfied and let $ p(x) \equiv 2 $ for all $x \in \close$.
	Then, problem \eqref{problem} admits a unique weak solution.
\end{theorem}

\begin{proof}
	Let $u,v\in \Wpzero{\mathcal{H}}$ be two weak solutions of \eqref{problem}. Testing the corresponding weak formulations with $\ph=u-v$ and subtracting these equations gives
	\begin{align}\label{uniqueness_1}
		\begin{split}
			& \into |\nabla (u-v)|^2\diff x + \into \mu(x) \left(|\nabla u|^{q(x)-2}\nabla u-|\nabla v|^{q(x)-2}\nabla u\right)\cdot \nabla (u-v)\diff x\\
			& = \into (f(x,u,\nabla u)-f(x,v,\nabla u))(u-v)\diff x+ \into (f(x,v,\nabla u)-f(x,v,\nabla v))(u-v)\diff x.
		\end{split}
	\end{align}
	The second term on the left-hand side of \eqref{uniqueness_1} is nonnegative, so we have the estimate
	\begin{align}\label{uniqueness_2}
		\begin{split}
			& \into |\nabla (u-v)|^2\diff x + \into \mu(x) \left(|\nabla u|^{q(x)-2}\nabla u-|\nabla v|^{q(x)-2}\nabla u\right)\cdot \nabla (u-v)\diff x\\
			& \geq \into |\nabla (u-v)|^2 \diff x.
		\end{split}
	\end{align}
	For the right-hand side of \eqref{uniqueness_1} we can use the assumptions \textnormal{(U1)}, \textnormal{(U2)} and H\"older's inequality which leads to
	\begin{align}\label{uniqueness_3}
		\begin{split}
			& \into (f(x,u,\nabla u)-f(x,v,\nabla u))(u-v)\diff x+ \into (f(x,v,\nabla u)-f(x,v,\nabla v))(u-v)\diff x\\
			& \leq  c_1\|u-v\|_2^2 +\into\left(f\left(x,v,\nabla \left(\frac{1}{2}(u-v)^2 \right)\right)-\rho(x) \right)\diff x\\
			&\leq c_1\|u-v\|_2^2+c_2\into |u-v| |\nabla (u-v)|\diff x\\
			& \leq\left(c_1\lambda_{1,2}^{-1}+c_2\lambda_{1,2}^{-\frac{1}{2}}\right)\|\nabla (u-v)\|_{2}^2.
		\end{split}
	\end{align}
	From \eqref{uniqueness_1}, \eqref{uniqueness_2} and \eqref{uniqueness_3} 
we see that
	\begin{align}\label{uniqueness_4}
		\begin{split}
			& \|\nabla (u-v)\|_{2}^2=\into |\nabla (u-v)|^2 \diff x \leq \left(c_1\lambda_{1,2}^{-1}+c_2\lambda_{1,2}^{-\frac{1}{2}}\right)\|\nabla (u-v)\|_{2}^2.
		\end{split}
	\end{align}
	Then, by \eqref{condition_coeffizients2}, from \eqref{uniqueness_4} it follows $u=v$.
\end{proof}

\section*{Acknowledgments}
The first author was funded by the Deutsche Forschungsgemeinschaft (DFG, German Research Foundation) under Germany's Excellence Strategy – 
The Berlin Mathematics Research Center MATH+ and the Berlin Mathematical School (BMS) (EXC-2046/1, project ID:\\ 390685689).

\end{document}